\documentclass[11pt]{amsart}
\usepackage{amsmath}
\usepackage{amssymb}
\usepackage{amsthm}
\numberwithin{equation}{section}

\usepackage{geometry}
\usepackage{xcolor}
\usepackage{hyperref}
\hypersetup{
    unicode=false,        
    pdftoolbar=true,      
    pdfmenubar=true,       
    pdffitwindow=false,     
    pdfstartview={FitH},    
    pdftitle={Geometric smoothing of the K\"ahler-Ricci flow},    
    pdfauthor={Di Nezza, Guedj, Lu},     
    colorlinks=true,       
   linkcolor=black,          
    citecolor=black,        
    filecolor=black,      
    urlcolor=black}

\newtheorem{theorem}{Theorem}[section]
\newtheorem{lem}[theorem]{Lemma}
\newtheorem{prop}[theorem]{Proposition}
\newtheorem{defi}[theorem]{Definition}

\newtheorem{notation}[theorem]{Notation}

\newtheorem{exa}[theorem]{Example}

\newtheorem{remark}[theorem]{Remark}
\newtheorem{quest}[theorem]{Question}

\newtheorem*{ackn}{Acknowledgements}

\newtheorem*{thmA}{Theorem A} 
 \newtheorem*{thmB}{Theorem B} 
\newtheorem*{thmC}{Theorem C} 
\newtheorem*{thmD}{Theorem D}

\newcommand{\PSH}{\textup{PSH}}

\newcommand{\B}{\mathbb B}
 \newcommand{\D}{\mathbb D}
 \newcommand{\R}{\mathbb R}
 \newcommand{\Q}{\mathbb Q}
 \newcommand{\C}{\mathbb C}
  \newcommand{\PP}{\mathbb P}
 \newcommand{\N}{\mathbb N}

  \newcommand{\f}{\varphi}
 
 \newcommand{\p}{\psi}

  \newcommand \e {\varepsilon}
 \newcommand \la {\lambda}
 
 \newcommand{\tr}{\operatorname{tr}}
 \newcommand{\Ric}{\operatorname{Ric}}

\begin{document}

\title{Geometric smoothing by the K\"ahler-Ricci Flow}

\date{\today}

\author{Eleonora Di Nezza}

\address{Universit\`a di Roma TorVergata, Dipartimento di Matematica Via della Ricerca Scientifica 1, 00133 Roma, Italy.}
\email{\href{mailto:dinezza@mat.uniroma2.it}{dinezza@mat.uniroma2.it}}
%\urladdr{\href{https://perso.pages.math.cnrs.fr/users/eleonora.di-nezza/}{https://perso.pages.math.cnrs.fr/users/eleonora.di-nezza/}}

\author{Vincent Guedj}
\address{Institut Universitaire de France \& Institut de Mathématiques de Toulouse\\
  Université de Toulouse; CNRS, UPS\\
  118 route de Narbonne, F-31400 Toulouse\\
  France}
  \email{\href{vincent.guedj@math.univ-toulouse.fr}{vincent.guedj@math.univ-toulouse.fr}}
%\urladdr{\href{https://www.math.univ-toulouse.fr/~guedj/}{https://www.math.univ-toulouse.fr/~guedj/}}

\author{Chinh H. Lu}
\address{Institut Universitaire de France \& Université d'Angers, CNRS, LAREMA, SFR MATHSTIC, F-49000 Angers, France}
\email{\href{hoangchinh.lu@univ-angers.fr}{hoangchinh.lu@univ-angers.fr}} 
%\urladdr{\href{https://math.univ-angers.fr/~lu/}{https://math.univ-angers.fr/~lu/}}

 \begin{abstract}
We study the geometric regularization  
of a positive closed current  by the (twisted) K\"ahler-Ricci flow
on a compact K\"ahler manifold. We conjecture that 
the local Arnold multiplicities 
linearly decrease  to zero,
while the flow produces complete K\"ahler metrics in the 
Zariski open subset of points that have small Lelong numbers.
We prove this conjecture in complex dimension 1 and provide several partial results in higher dimension.
 \end{abstract}

\maketitle

\setcounter{tocdepth}{1}
\tableofcontents

\section*{Introduction}

Regularization of positive closed currents  plays a pivotal role
in complex analysis and geometry. Demailly has produced over the last decades several fundamental 
regularization results of analytic nature, using convolutions and Bergman kernel approximations (see \cite{Dem92,DPS01,DK01}).
Our main goal in this article is to propose an alternative geometric regularization process, 
by using the (twisted) Ricci flow.

\smallskip

In the whole article we let $(X,\omega)$ be a compact K\"ahler manifold of dimension $n$.
We fix $\f_0$ an $\omega$-plurisubharmonic function and
consider the complex Monge-Ampère flow 
\[
(\omega+dd^c \varphi_t)^n = e^{\dot \varphi_t} dV_X
\]
starting from $\f_0$. Here $\varphi_t$ denotes the ``maximal weak solution" to the flow.

As shown in \cite{GZ17,DNL17}, $\f_t$ is smooth in a Zariski open set $\Omega_t$ where it solves the equation in the classical sense,
and it has logarithmic singularities at the boundary of this set.
Among all possible solutions to the flow in $\Omega_t$, $\f_t$ is the unique one having minimal singularities
along $\partial \Omega_t$.
Equivalently $\omega_t=\omega+dd^c \f_t$ is a K\"ahler form in $\Omega_t$
which is a solution of the twisted K\"ahler-Ricci flow
\begin{equation*} \label{eq:KR}
\tag{KR}
\frac{\partial \omega_t}{\partial t}=-{\rm Ric}(\omega_t)+{\rm Ric}(\omega),
\end{equation*}
and whose potential $\f_t$ has the smallest singularities along $\partial \Omega_t$.
In this article we conjecture that the flow \eqref{eq:KR}
gradually replaces the analytic singularities of $\f_0$ by milder ones, producing a 
complete K\"ahler metric $\omega_t$ in $\Omega_t$.

In our previous work \cite{DGL26} we have analyzed the case when $\omega+dd^c \f_0$
has log smooth divisorial singularities, showing that the latter are linearly replaced 
by Poincaré type ones. 
We develop here the first steps of a similar analysis for an arbitrary current $\omega+dd^c \f_0$.
We first treat the case of complex dimension $n=1$.
%when $X$ is a compact Riemann surface.

 \begin{thmA} \label{thm:A}
Let $(X,\omega)$ be a compact Riemann surface, and let 
$\mu_0= \sum_{j \geq 0} m_j \delta_{a_j} +R_0$
be a positive Radon measure, where
$m_j \delta_{a_j}$ denotes a Dirac mass of size $m_j>0$ at $a_j \in X$, 
and $R_0$ is a positive Radon measure with no atom.
The (twisted) Ricci-flow $(\omega_t)_{t>0}$   emanating from $\mu_0$
can be decomposed, for $0<t<\la(\mu_0)$, as
$$
\omega_t=\sum_{m_j>t} (m_j-t) \delta_{a_j}+\beta_t,
$$
where $\beta_t$ is a complete metric in $X \setminus  A_t$, with
$A_t=\{a_j \text{ such that } m_j>t \}$.

Moreover $\omega_t$ is a global K\"ahler form when $t > \la(\mu_0)=\max_j m_j$.
 \end{thmA}
 
 This result complements the findings of Giesen and Topping \cite{GT11,Top15},
who constructed  the Ricci flow on open  Riemann surfaces and studied the way
it provides instantaneously complete K\"ahler metrics.
The precise asymptotic of the complete metrics $\beta_t$ depends on the nature of the singularities
of $R_0$ as we explain in Section \ref{sec:exampledim1}; 
let us stress in particular that $\beta_t$ is not necessarily of Poincaré type.

 \smallskip
 
 We conjecture that a  similar geometric smoothing holds in higher dimension as well.
 This requires at first to understand the evolution of the logarithmic singularities of
 $\omega_t$ along the K\"ahler-Ricci flow. This is the contents of our second main result:
 
  \begin{thmB} \label{thm:B}
Let $(X,\omega)$ be a compact K\"ahler manifold with $\dim_\C X=n\geq 1$, and  let $\omega_t=\omega+dd^c \f_t$ be the solution of the K\"ahler-Ricci flow \eqref{eq:KR} starting at $\omega+dd^c \f_0$.
Then for all $(t,x) \in \R^+ \times X$,
$$
\la(\f_t,x) \geq \max(\la(\f_0,x)-t,0)
\; \; \text{ and } \; \;
\la(\f_t,x) \leq \max \left( \la(\f_0,x)-t \frac{\la(\f_0,x)}{\la(\f_0)}, 0 \right).
$$
Thus $\la(\f_t)=\max(\la(\f_0)-t,0)$ and 
 $\omega_t=\omega+dd^c \f_t$ is a K\"ahler form on $X$  when $t>\la(\f_0)$.
 \end{thmB}
 
 Here $\la(\f,x)$ denotes the Arnold multiplicity of $\f$ at the point $x$, i.e. the inverse of the integrability index $c(\f,x)$, and $
\la(\f)=\sup \{ \la(\f,x), \; x \in X \}
$.      It follows from Skoda's exponential integrability that
 $$
\frac{\nu(\f,x)}{n} \leq  \la(\f,x) \leq \nu(\f,x),
 $$
 where $\nu(\f,x)$ denotes the Lelong number of $\f$ at $x$. 
  
 We expect that $\la(\f_t,x) =\max \left( \la(\f_0,x)-t, 0 \right)$; this is true in dimension one, as well as in all the examples that we have analyzed.
 If such is the case, it would then follow 
 %from the upper semi-continuity of   $x \mapsto \la(\f_0,x)$ 
 that there exists a (possibly finite) decreasing sequence of jumping times $(t_j)$ 
for the family of decreasing analytic subsets
 $t \mapsto E_t:=\{ x \in X, \, \la(\f_0,x) \geq t \}$,  with  $t_0=\la(\f_0)$
 and $\lim_j t_j=0$.
 We set $\Omega_{\ell}=X \setminus E_{t_{\ell}}$
 and conjecture that
 \begin{itemize}
 \item  $\omega_t$ is a complete K\"ahler metric in $\Omega_{\ell}$ for $t \in (t_{1+\ell},t_{\ell})$;
 \item  $\omega_t$ is a  K\"ahler current on $X$ for $t \in (t_{1+\ell},t_{\ell})$, i.e. $\omega_t \geq \delta_{\ell} \omega$ for some $\delta_{\ell}>0$.
 \end{itemize} 
 We establish partial results and treat particular cases that support this conjecture.
 
 %\medskip 
 
   It follows from \cite[Theorem 3.2]{DNL17} that for $\e>0$ fixed,
 $$
 \f_{\e} \geq [1-\lambda(\f_0) \e] \p_{\e}-C_{\e},
 $$
 where $\p_{\e}$ is a quasi-psh function with analytic singularities. If
 $\omega=\Theta_h$ is the curvature of a hermitian metric of an ample line bundle $L$, one can take
 $\p_{\e}$ to be the Bergman kernel of 
 the Hilbert space
 $
 {\mathcal H}_j=\left \{ s \in H^0(X,L^j), \; \int_X |s|^2 e^{-2j\f} dV_X < +\infty \right \},
 $
 with $j=[1/\e]$.
   Using the semi-group property, we can thus assume that $\f_0$ is bounded below
 by a function with analytic singularities. It is thus natural to consider 
 the case when $\f_0$ itself has analytic singularities,
 this is the setting of our third main result (Theorem \ref{thm:c2hyp}).

 \begin{thmC} \label{thm:C}
 Assume $\f_0$ has analytic singularities. Let   $\pi:Y \rightarrow X$ be a log resolution, with
 $ \pi^* (\omega+dd^c \f_0)=\sum_{j=1}^q m_j [D_j]+R_0$, 
where $m_j \in \R^+$,   $D=\sum_j D_j$ is log smooth,  and $R_0$ a semi-positive closed form.
If the cohomology class $\{ \lambda(\f_0) \Theta +  R_0 \}$ is K\"ahler then
$$
\pi^*\omega_t= \sum_{j=1}^q (m_j-(1+b_j)t)[D_j]+ \beta_t,
$$
for $0<t<t_1(\f_0)=\min_{1 \leq j \leq q} \frac{m_j}{1+b_j}$,
where $\beta_t$ is a complete K\"ahler metric in $\pi^{-1}(\Omega_t)$.
 \end{thmC}
 
 Here $\Theta=\sum_{j=1}^q (1+b_j) \Theta_j$, where the $b_j$'s are the discrepancies of 
 the resolution and $\Theta_j$ is the curvature of a fixed hermitian metric of ${\mathcal O}(D_j)$.
We show in Section \ref{sec:logresolution} that the class $\{ \lambda(\f_0) \Theta + R_0 \}$ is
always big, and we exhibit examples where it is either K\"ahler (Section \ref{sec:cone})
or not even nef  (Section \ref{sec:cubics}).
 We recall in Section \ref{subsec_analytic sing} that
 $\lambda(\f_0)=\max_{1 \leq j \leq q} \frac{m_j}{1+b_j}$.

 \smallskip
 
 In the final Section \ref{sec:isolated} we consider  initial data $\f_0$ with isolated singularities.
 We analyze a large family  with toric symmetries, describing the evolution of their logarithmic singularities
 along the flow, as well as a quasi-sharp ${\mathcal C}^0$-estimate (Proposition  \ref{pro:c0toric}). 
 This relies on a result of independent interest (Lemma \ref{lem:closeopen}),
 which shows that one can always recover integrability of quasi-psh functions at their critical
 exponent by adding a finite energy weight.
 
The higher order description requires one to have a model metric with good curvature properties.
This is the case when the singularity is {\it homogeneous}, in which case
$\beta=dd^c \log|z|^2+dd^c (-2\log(-\log|z|^2))$ is such a metric  in a local chart near $0 \in \C^n$.

  \begin{thmD} \label{thm:D}
  Assume the $\omega$-psh function $\f_0$ has an isolated homogeneous singularity at 
  some point $a \in X$.
  % i.e.  $\f_0=\log|z|^2+u$ in a local chart such that $a=0 \in \C^n$,  where $u$ is smooth.
  Then $\lambda(\f_0)=1/n$ and for all $t \in (0,1/n)$,
  $$
  \f_t=(1-nt)\f_0-2t\log(-\f_0)+v_t,
  $$ 
  where $v_t$ is a bounded quasi-psh function.
  
  Moreover  $e^{-C/t}  \beta \leq \omega_{t} \leq e^{C/t} \beta$ for some $C>0$, hence
$\omega_{t}$ is a complete metric in $X\setminus \{a\}$.
 \end{thmD}

We finally develop the first steps of the analysis of initial data which have 
a combination of homogeneous isolated and divisorial singularities
(see Proposition \ref{pro:c0combine} and Section \ref{sec:combineC2}).

\begin{ackn} 
The authors are partially supported by the Institut Universitaire de France
and the Fondation Charles Defforey. 
EDN is supported by the ERC grant SiGMA No. 101125012 and the MUR Excellence  Project MatMod@TOV CUP E83C2300\-0330006. 
VG is partially supported by the Clay Foundation. C.H.Lu is partially supported by the Centre Henri Lebesgue ANR-11-LABX-0020-01.
This material is based upon work supported by the 
%National Science Foundation under Grant No. DMS-1928930,
NSF  Grant No. DMS-1928930, 
while the authors were in residence at 
%the Simons Laufer Mathematical Sciences Institute 
SLMath
in Berkeley, California, Fall 2024, as part of the Special Geometric Structures and Analysis program. 
\end{ackn}

 %\vfill
% \pagebreak[4]

\section{Arnold multiplicity} \label{sec:arnold}

In the whole article we let $X$ be a compact K\"ahler manifold of dimension $n$.

\subsection{Definition of the flow}

Fix $\omega$ a K\"ahler form.
 We recall here the construction of the maximal solution to the K\"ahler-Ricci flow \eqref{eq:KR},
following \cite{GZ17,DNL17,DGL26}.

  \subsubsection{Plurisubharmonic functions}

  A function is quasi-plurisub\-harmonic (quasi-psh) if it is locally given as the sum of  a smooth and a plurisubharmonic function.   
% \textcolor{red}{I do not agree of the choice of $u$ as notation. Later in the text $u$ will be the potential of the logarithmic part of the flow. Later we use $\ell_0$}
\begin{defi}
Quasi-psh functions
$\f:X \rightarrow \R \cup \{-\infty\}$ satisfying
$
\omega_{\f}:=\omega+dd^c \f \geq 0
$
in the weak sense of currents are called $\omega$-plurisubharmonic ($\omega$-psh for short).

A quasi-psh function $\f$ has \emph{analytic singularities}
 if locally  
$
\f=\frac{1}{2m}\log\sum_{j=1}^N|f_j|^2+u,
$
for some holomorphic functions $f_1,\dots,f_N$, $m\in\N^*$ and smooth function $u$. 
When $N=1$ we say that $\f$ has \emph{divisorial singularities}.
\end{defi}

 We let $\PSH(X,\omega)$ denote the set of all $\omega$-plurisubharmonic functions which are not identically $-\infty$.  
 If $\f$ is a quasi-psh function with analytic singularities, there exists a modification $\pi\colon Y\to X$, isomorphic over 
 $\{\f>-\infty\}$, such that $\pi^*\f$ has divisorial singularities.

\begin{defi}
Given a quasi-psh function $\f$ on $X$, we let 
$$
c(\f):=\sup \left\{c>0, \int_X e^{-c\f} dV_X <+\infty \right\}
$$
denote the integrability index of $\f$, while $\la(\f)=c(\f)^{-1}$ is the Arnold multiplicity of $\f$.

We define similarly $c(\f,x)$ and $\la(\f,x)=c(\f,x)^{-1}$ the local version of these invariants,
and let $\nu(\f,x)=\sup \{ \gamma \geq 0, \; \f(y) \leq \gamma \log {\rm d}_{\omega}(x,y)+C \}$
be the Lelong number of $\f$ at $x \in X$.
%where ${\rm d}_{\omega}$ denotes the Riemannian distance induced by $\omega$.
\end{defi}

It follows from Skoda's integrability theorem \cite[Theorem 2.50]{GZbook} that for all $x \in X$,
$$
\frac{\nu(\f,x)}{n} \leq \la(\f,x) \leq \nu(\f,x).
$$
We refer the reader to \cite{DK01,GZbook} for basic properties of these numerical invariants, 
as well as their counterpart in algebraic geometry. We simply stress here that
$$
\la(\f)=\sup \{ \la(\f,x), \; x \in X \},
$$
while $c(\f)=\inf \{ c(\f,x), \; x \in X \}$,
%and 
that 
%$(\f,x) \mapsto \nu(\f,x)$   and 
$\f \mapsto \la(\f,x)$ is upper semi-continuous 
and that $E_t(\f)=\{ x \in X, \, \lambda(\f,x) \geq t \}$ is a closed analytic subset
(see \cite[Main theorem and Proposition 1.4.1]{DK01}).

\subsubsection{Construction} \label{sec:construction}

Fix $\f_0 \in \PSH(X,\omega)$ and set 
$
\Omega_t=\{x \in X, \nu(\f_0,x)<t \}.
$
The latter form an increasing sequence of Zariski open subsets, with
$\Omega_t=X$ as soon as $t>\max_{x \in X} \nu(\f_0,x)$.
We can approximate $\f_0$ by a decreasing
sequence $(\f_{0,j})$ of smooth $\omega$-psh functions.
Let $\f_{t,j}$ denote the unique smooth solutions of 
$$
(\omega+dd^c \varphi_{t,j})^n = e^{\dot \varphi_{t,j}}dV_X,
$$
starting from   $\f_{0,j}$. It is shown in \cite{GZ17,DNL17} that
for all $(t,x) \in \R^+ \times X$, 
\begin{itemize}
\item the sequence  $j \mapsto \f_{t,j}(x)$ decreases to $\f_t(x) \in \R \cup \{-\infty\}$;
\item the limit $\f_t(x)$ is independent of the choice of approximants;
\item the function $x \mapsto \f_t(x)$ belongs to $\PSH(X,\omega)$;
\item the function $(t,x) \mapsto \f_t(x)$ is smooth in $\R_*^+ \times \Omega_t$, and
$$
(\omega+dd^c \varphi_{t})^n = e^{\dot \varphi_{t}}dV_X
\; \; \text{ in } \; \;
\Omega_s, \; s<t;
$$
\item the function $\f_t$ may have positive Lelong numbers at some points in $\partial \Omega_t$;
\item one has $\f_t \rightarrow \f_0$ in $L^1$ as $t \rightarrow 0$.
\end{itemize}

It is moreover shown in \cite[Proposition 1.15]{DGL26} that
$t \mapsto \phi_t= \f_t -n(t\log t-t)$ is concave with $\dot{\phi}_t \geq -C$.
Thus the convergence $\f_t \rightarrow \f_0$ is actually quite strong, in particular it holds in
the sense of capacity (see \cite[Section 4.2.2]{GZbook}).

\subsection{Maximum principles}

\subsubsection{Sub/super-solutions}

\begin{defi}
An upper semicontinuous function $\f: \mathbb R^+ \times X \rightarrow \R \cup \{-\infty\}$ is called a subsolution  of the flow if 
\begin{itemize}
\item $\f_t(\cdot)=\f(t,\cdot)$ belongs to $\PSH(X,\omega)$ for all $t \in \R^+$; 
\item $(t,x) \mapsto \f_t(x)$ is smooth in $R_*^+  \times \Omega$ for
some Zariski dense open set $\Omega \subset X$, and
$$
(\omega+dd^c \varphi_{t})^n \geq e^{\dot \varphi_{t}}dV_X
\; \; \text{ in } \; \;
\Omega.
$$
\end{itemize}
\end{defi}

One defines similarly the notion of supersolution, and the notion of (weak) solution: 

 \begin{defi}
A function $\f: \R \times X  \rightarrow \R \cup \{-\infty\}$ is  a supersolution of the flow if 
 $(t,x) \mapsto \f_t(x)$ is smooth in $\R^+_* \times \Omega$ for
some Zariski dense open set $\Omega \subset X$, and for  $t>0$,
$$
{\bf 1}_{\{\omega +dd^c \varphi_t \geq 0\}} (\omega+dd^c \varphi_{t})^n \leq e^{\dot \varphi_{t}}dV_X
\; \; \text{ in } \; \;
\Omega.
$$

A function  is   a weak solution of the flow
if it is both a subsolution and a supersolution.
\end{defi}

An important role is  played by the following property (see \cite[Theorem 5.8]{DNL17}):

\begin{theorem} \label{thm:maxflow}
The solution 
%of the flow
 constructed in Section \ref{sec:construction} is the envelope of
all subsolutions to the flow. In particular it has minimal singularities among all
weak solutions to the flow.
\end{theorem}

In the rest of this article we call this solution {\it the} solution
%(or the maximal solution) 
of the flow.

\smallskip
 
 This maximality property has the following useful consequence, called the semi-group property:
fix $s>0$, let $\f_t$ be the solution of the flow starting from $\f_0$,
and let $\p_t$ denote the solution of the flow starting from  $\f_s$.
Then $\p_t=\f_{t+s}$ for all $t \geq 0$.

\subsubsection{Comparison principle}

Let $\omega$ be a semipositive and big form. 
Fix $\Omega$ a Zariski open set,
and $\rho$ a strictly $\omega$-psh  function which is smooth in $\Omega$
 with $\partial \Omega = \{\rho=-\infty\}$.
The following result is an important variation on the classical parabolic maximum principle (see \cite[Theorem 1.13]{DGL26}):

\begin{theorem}\label{thm:pcp}
Let $\varphi_t$ (resp. $\psi_t$) be a subsolution (resp. supersolution) of the flow 
\[
(\omega+dd^c u_t)^n =e^{\dot{u}_t +f} dV_X, 
\]
 with 
$\varphi,\psi,f$ smooth in $(0,T] \times  \Omega$. 
Assume the following properties hold:
\begin{enumerate}
\item $\psi$ is continuous on $[0,T]\times \Omega$ and $\varphi_0\leq \psi_0$;
\item there is $a>0$, $h\in \mathcal{E}(X,a\omega)$  such that $\varphi_t(x)+h(x)\leq \psi_t(x)$, 
for $(t,x) \in [0,T] \times \Omega$;
\item $\rho(x) \leq \psi_t(x)$, for all $(t,x) \in [0,T] \times \Omega$.
\end{enumerate}
Then $\f_t(x) \leq \p_t(x)$ for all $(t,x) \in [0,T] \times \Omega$.
\end{theorem}
 
The class $ \mathcal{E}(X,\omega)$ is the set of all $\omega$-psh functions with full Monge-Amp\`ere mass \cite{GZ07}. 
Functions in $\mathcal{E}(X,\omega)$ have mild singularities, in particular
they have zero Lelong numbers.
We emphasize that \cite[Theorem 1.13]{DGL26} is stated for $a=1$, 
but a very similar proof yields the above slightly more flexible statement.

\subsection{Flowing analytic singularities}\label{subsec_analytic sing}

When $\f_0$ has analytic singularities, one can use a log resolution $\pi:Y \rightarrow X$ to resolve the latter, i.e.
$$
\pi^*(\omega+dd^c \f_0)=\sum_{j=1}^q m_j [D_j]+R_0,
$$
 where $R_0 \geq 0$ is a smooth form, $m_j > 0$,
and $D_j$ are smooth divisors with simple normal crossings.
We let $\sigma_j$ be a holomorphic defining section  $D_j$, and fix $h_j$ a smooth metric of 
the line bundle ${\mathcal O}(D_j)$.
%\smallskip
We let $b_j \in \N$ denote the discrepancies of this resolution, defined by 
$$
\pi^* dV_X = \prod_{j=1}^q |\sigma_j|^{2 b_j} dV_Y,
$$
where $dV_Y$ denotes a volume form on $Y$.
It is classical (see \cite[Proposition 1.7]{DK01}) that 
$$
c(\f_0)=\min_{1 \leq j \leq q} \frac{1+b_j}{m_j}.
$$
This follows from the fact that $\pi^* \left( e^{-c \f_0}   dV_X \right)$ is equivalent
 to $ \prod_{j=1}^q |\sigma_j|^{2b_j-2m_jc} dV_Y$.
A local version of the above equality holds as well (see \cite[Appendix B]{BBJ21}), i.e. for any $x\in X$,
\begin{equation}\label{int formula}
 c(\f_0, x)=\min\left\{ \frac{1+b_j}{m_j}, \; j \; :\;  \pi(D_j)=x \right\}.
 \end{equation}

\begin{lem} \label{lem:analytsing}
If $\f_0$ has analytic singularities, then for all $0<t<\lambda(\f_0)$ one has 
$$
\f_t \circ \pi \leq \sum_{j=1}^q \max(m_j-(1+b_j)t,0) \log |\sigma_j|^2_{h_j}+Ct
$$ 
for some constant $C>0$. In particular  $\lambda(\f_t,x) \geq \lambda(\f_0,x)-t$ for all $(t,x)$.
\end{lem}

\begin{proof}
Let $\omega_Y$ be a K\"ahler form on $Y$ and consider the approximating sequence 
\[
 v_{t,\varepsilon} = \sum_{j=1}^q \left ((m_j-t) \log (|\sigma_j|^2+\varepsilon^2) - b_jt \log|\sigma_j|^2\right).
\] 
Writing $\log (|\sigma_j|^2+\varepsilon^2)=\chi \circ \log |\sigma_j|^2$ with $\chi(x)=\log(e^x+\varepsilon^2)$, and differentiating,
we obtain on $Y\setminus \cup_j D_j$,
\begin{eqnarray*}
dd^c v_{t,\varepsilon} 
&\leq& \sum_{j=1}^q \left( C_1 \omega_Y+C_1 \frac{|\sigma_j|^2 \omega_Y}{(|\sigma_j|^2+ \varepsilon^2)}  
+  \frac{i}{\pi}  \frac{\varepsilon^2 \partial \sigma_j \wedge \overline{\partial \sigma_j}  }{(|\sigma_j|^2+\varepsilon^2)^2}\right) \\
&\leq& C_2 \left( \omega_Y + \sum_{j=1}^q \frac{i}{\pi} \frac{\partial \sigma_j \wedge \overline{\partial \sigma_j} }{|\sigma_j|^2+\varepsilon^2} \right),
\end{eqnarray*}
using $\varepsilon^2 \leq |\sigma_j|^2+\varepsilon^2$ and $|\sigma_j|^2 \leq |\sigma_j|^2+\varepsilon^2$.
 Observe that
 $$
 \left( \omega_Y + \sum_{j=1}^q \frac{i}{\pi} \frac{\partial \sigma_j \wedge \overline{\partial \sigma_j} }{|\sigma_j|^2+\varepsilon^2} \right)^n
 \leq C_3 \prod_{j=1}^q \frac{1}{|\sigma_j|^2+\varepsilon^2} \omega_Y^n \leq C_3' \prod_{j=1}^q \frac{1}{|\sigma_j|^2+\varepsilon^2} dV_Y,
 $$
 since the  forms $\frac{\partial \sigma_j \wedge \overline{\partial \sigma_j} }{(|\sigma_j|^2+\varepsilon^2)}$ have rank $1$.
 
 \smallskip
 
At the points where $\pi^* \omega+dd^c v_{t,\varepsilon} \geq 0$, 
we infer
\[
(\pi^* \omega+dd^c v_{t,\varepsilon})^n 
\leq C_3' \prod_{j=1}^q \frac{1}{|\sigma_j|^2+\varepsilon^2} dV_Y
= e^{\partial_t {v}_{t,\varepsilon}+\log C_3'} \pi^* dV_X.
\]

Thus $v_{t,\e}+t \log C_3'$ is a super-solution of the flow
$(\pi^* \omega+dd^c \p)^n =e^{\partial_t \p} \pi^* dV_X$
in $Y$ with initial data $\pi^*\f_0$.
We can then apply Theorem \ref{thm:pcp} with $\Omega= X\setminus \bigcup_j D_j$ to get 
$$\f_t \circ \pi  \leq v_{t,\varepsilon} + t \log C_3',
$$
and by letting $\varepsilon\to 0$ we obtain
\[
\f_t \circ \pi \leq \sum_{j=1}^q (m_j-(1+b_j)t) \log |\sigma_j|^2 +   C_3't. 
\]
Using the upper bound $\varphi_t\circ \pi \leq C_4$ for $t\in [0,\lambda(\varphi_0)]$ and Lemma \ref{lem: Siu} below we arrive at the conclusion. For the last statement we observe that by \eqref{int formula} and the above inequality we have $ c(\f_t, x)\leq \min_j  \frac{1+b_j}{m_j -(1+b_j)t},$ where the minimum is taken over $j$ such that $m_j -(1+b_j)t \geq 0$ and $\pi(D_j)=x$. Hence 
$$\lambda (\f_t, x) \geq \max_{1\leq j\leq q} \frac{m_j -(1+b_j)t}{1+b_j}= \max_{1\leq j\leq q} \frac{m_j }{1+b_j} -t= \lambda (\f_0, x)-t.$$
\end{proof}

\begin{lem}\label{lem: Siu}
	If $u$ is a quasi-psh function  such that 
	$
	u\leq \sum_{j=1}^{q} c_j \log |\sigma_j|^2, \; c_j \in \mathbb R.
	$
	Then 
	$$
	u \leq \sum_{j=1}^{q} \max(c_j,0) \log |\sigma_j|^2 +C, 
	$$
	for some constant $C$ that only depends on an upper bound on $\sup_X u$.
\end{lem}

\begin{proof}
Set $v= u +\sum_{j=1}^{q} \max(-c_j,0)\log |\sigma_j|^2$.
This is a sum of quasi-plurisubharmonic functions, hence
we can find a K\"ahler form $\omega_Y$ such that  $v\in \PSH(Y,\omega_Y)$.
By Siu's decomposition theorem, we have 
	\[
	\omega_Y + dd^c v \geq  \sum_{j=1}^{q} \max(-c_j,0) [\sigma_j=0] . 
	\]
	We infer $v\leq  \sum_{j=1}^{q} \max(-c_j,0) \log |\sigma_j|^2+C$ and the conclusion follows. 
	\end{proof}

\subsection{Singularities decrease linearly}

Recall that 
$\la(\f_0)=\sup \{ \la(\f_0,x), \; x \in X \}=0$ if and only if
$\f_0$ has zero Lelong numbers at all points. As this case has  been treated in \cite{GZ17,DNL17}, 
we assume throughout this article that $\la(\f_0)>0$.

\begin{theorem} \label{thm:arnold}
%Assume that $\la(\f_0)>0$ and 
Let $\omega_t=\omega+dd^c \f_t$ be the solution of the K\"ahler-Ricci flow
starting at $\omega+dd^c \f_0$.
Then for all $(t,x) \in \R^+  \times X$,
$$
\la(\f_t,x) \geq \max(\la(\f_0,x)-t,0)
\; \; \text{ and } \; \;
\la(\f_t,x) \leq \max \left( \la(\f_0,x)-t \frac{\la(\f_0,x)}{\la(\f_0)}, 0 \right).
$$
%In particular 
Thus $\la(\f_t)=\max(\la(\f_0)-t,0)$ and 
 $\omega_t$ 
 is a K\"ahler form on $X$ precisely when $t>\la(\f_0)$.
\end{theorem} 

Thus the Arnold multiplicity 
%decreases affinely along the K\"ahler-Ricci flow, hence  
determines the finite time   after which the flow becomes smooth.
 This largely generalizes \cite[Lemma 4.4]{DNL17}.

\begin{proof}
It follows from Demailly's equisingularity approximation result 
(see \cite[Theorem 2.2.1]{DPS01})
that one can find a sequence $(\f_{0,k})_k$ of $\omega$-psh functions with 
analytic singularities such that $\f_0 \leq \f_{0,k}$ and
$c(\f_0, x)=\lim_{k\rightarrow +\infty} c(\f_{0,k}, x)$ for any $x$. The maximum principle ensures that
the flows $\f_{t,k}$ emanating from $\f_{0,k}$ satisfy
$\f_t \leq \f_{t,k}$; in particular $\lambda(\f_t,x) \geq \lambda(\f_{t,k},x)$. It follows therefore from Lemma \ref{lem:analytsing} that
$$
\lambda(\f_t,x) \geq \lambda(\f_{t,k},x)\geq \lambda(\f_{0,k},x)-t.
$$
Letting $k \rightarrow +\infty$ we infer $\la(\f_t,x) \geq \max(\la(\f_0,x)-t,0)$.

\smallskip

Fix $0<\alpha<c(\f_0)$. 
By \cite[Lemma 2.9]{GZ17} 
one can find $C_{\alpha}>0$ such that 
$$
(1-\alpha t) \f_0 +n (t \log t-t) -C_{\alpha}t \leq \f_t
$$
for all $x \in X$ and $0<t<1/\alpha$. It then follows that 
$$
\lambda(\f_t, x) \leq \lambda( (1-\alpha t) \f_0, x)= (1-\alpha t) \lambda(\f_0, x).
$$
 Letting $\alpha$ converging to $c(\varphi_0)$ we obtain $\lambda(\f_t, x) \leq \lambda(\f_0, x)-t\frac{\lambda(\f_0, x)}{\lambda(\f_0)}$. 
 Extremizing over $x \in X$ we obtain $\lambda(\f_t) \leq \max(\lambda(\f_0)-t,0)$.
\end{proof}

\begin{quest} \label{quest:arnoldlocal}
Does one have $\la(\f_t,x)=\max(\la(\f_0,x)-t,0)$ ?
\end{quest}

We provide below a positive answer in dimension 1 (see Lemma \ref{lem:singdim1}).

 \section{Compact Riemann surfaces}  \label{sec:riemann}

In this section we explain how to derive the proof of Theorem A.
Let $(X,\omega)$ be a compact Riemann surface, and let 
$\mu_0=\omega+dd^c \f_0$ be a positive Radon measure.
The measure $\mu_0$ has at most countably many atoms, so
we can decompose it as
$$
\mu_0= \sum_{j \geq 0} m_j \delta_{a_j} +R_0,
$$
where $R_0$ is a positive Radon measure with no atom,
$\delta_{a_j}$ denotes the Dirac mass at $a_j \in X$
and $m_j=\nu(\mu_0,a_j)$, the mass of $\mu_0$ at $a_j$, is 
precisely the Lelong number of $\mu_0$ at $a_j$.

\subsection{Dirac masses  decrease linearly in time}

We can assume without loss of generality that $\int_X \omega=1$,
 and that $\omega$ is a multiple of the curvature form of 
a hermitian metric $h$ of some very ample $\Q$-line bundle $L$ on $X$. 
The series $\sum_{j \geq 0} m_j$ converges as it is bounded from
above by the total mass of $\mu_0$, hence $m_j \rightarrow 0$.
Relabelling if necessary, we can assume  that
 $j \mapsto m_j$ is decreasing.
 We set 
 $$
  m=\sum_{j \geq 0} m_j
 \; \; \text{ and } \; \; 
 R_0=(1-m) \omega+dd^c \p_0, 
 $$
% \; \; \text{ where } \; \;
where $ \p_0 \in \PSH(X,(1-m) \omega)$.
We let $\sigma_j$ denote holomorphic sections of $L$ such that
$$
\delta_{a_j}=\omega+dd^c \log|\sigma_j|^2,
$$
where the norm $|\sigma_j|$ is computed with respect to the metric $h$.

 \begin{lem} \label{lem:singdim1}
 Fix $ j \in \N$ and $0<\e<1$. There exists $C_{\e,j}>0$ such that for all $t \in [0,\lambda(\f_0)]$,
$$
  \max(m_j- (1-\e) t,0) \log |\sigma_j|^2 +\p_j-C_{\e,j}t +(t \log t-t) \leq   \f_t,
  $$
  where $\p_j \in \PSH(X,(1-m_j) \omega)$ is such that  $\f_0=m_j \log |\sigma_j|^2+\p_j$.  
  In particular 
  $$
  \la(\f_t,x)=\max \left( \la(\f_0,x)-t,0 \right)
  \; \; \text{  for all }
  (t,x) \in \R^+ \times X.
  $$
 \end{lem}

\begin{proof}
Consider $u=u(\e,j) \in \PSH(X,\omega)$ the unique solution of 
$$
\omega+dd^c u=\frac{e^{(1-\e)u}}{|\sigma_j|^{2(1-\e)}} \omega.
$$
Observe that ${|\sigma_j|^{-2(1-\e)}}\in L^p (\omega)$ for $p\in (1, (1-\varepsilon)^{-1})$, 
hence $u$ is bounded on $X$. For $0 \leq t \leq m_j(1-\e)^{-1}$ we set
$$
v(t,x)=[m_j-(1-\e)t] \log |\sigma_j|^2+\p_j+(1-\e)t u+t \log(1-\e) +(t \log t -t).
$$
Observe that $v(0,x)=\f_0(x)$ and $e^{\partial_t v}=(1-\e)t e^{(1-\e)u} |\sigma_j|^{-2(1-\e)}$, hence
$$
\omega+dd^c v \geq (1-\e)t (\omega+dd^c u) =e^{\partial_t v} \omega.
$$
Thus $v$ is a subsolution in $[0, m_j(1-\e)^{-1}] \times X$. 
The desired lower bound on $\f_t$ follows, since $u$ is uniformly bounded.

\smallskip

Set $w_t=\p_j-C_{\e,j}t +(t \log t-t)$, $t\in [ m_j(1-\e)^{-1}, \lambda (\mu_0)]$.
The previous analysis shows that $w_s \leq \f_s$ for $s=m_j(1-\e)^{-1}$. 
Enlarging the constant $C_{\e,j}$ if necessary, we can assume that 
$m_j \geq te^{-C_{\e,j}} $ for all $t\in [0, m_j(1-\varepsilon)^{-1}]$. 
Hence $\omega+dd^c w_t \geq m_j \omega \geq te^{-C_{\e,j}} \omega =e^{\partial_t w} \omega$. The comparison principle and the semigroup property yield $w_t\leq \varphi_t$ in $[ m_j(1-\e)^{-1}, \lambda (\mu_0)]$.
It follows therefore that for all $t \in [0,\lambda (\mu_0)]$,
$$ 
\max(m_j- (1-\e) t,0) \log |\sigma_j|^2 +\p_j-C_{\e,j}t +(t \log t-t) \leq   \f_t.
$$

Since $\p_j$ has no logarithmic singularity at the point $a_j$, we infer 
$\nu(\f_t,a_j) \leq m_j -(1-\e)t$. Since $c(\varphi_0, a_j)=m_j^{-1}$ we get $\la(\f_t,a_j) \leq \nu(\f_t,a_j)\leq  \la(\f_0,a_j)-t$ by letting $\e \rightarrow 0$, while $\la(\f_t,x)=\la(\f_0,x)=0$ for  $x\neq a_j$.
The reverse inequality is provided by Theorem \ref{thm:arnold}.
\end{proof}

\subsection{Proof of Theorem A}
 
We now prove Theorem A, whose statement is as follows:

\begin{theorem} \label{thm:riemsurface}
The (twisted) Ricci-flow $(\omega_t)_{t>0}$   emanating from $\mu_0$
can be decomposed, for $0<t<\la(\mu_0)$, as
$$
\omega_t=\sum_{m_j>t} (m_j-t) \delta_{a_j}+\beta_t,
$$
where $\beta_t$ is a complete metric in $X \setminus  A_t$, where 
$A_t=\{a_0, \ldots, a_{j(t)} \}$ with $j(t)$ the greatest index such that $m_{j(t)}>t$.
Moreover $\omega_t$ is a global K\"ahler form when $t > \la(\mu_0)$.
\end{theorem}

When $\mu_0= \sum_{j=0}^q m_j \delta_{a_j}+R_0$ is a finite sum of atoms modulo a smooth remainder $R_0 \geq0$, 
this result follows from the 
main result of \cite{DGL26}, which moreover shows that 
$\beta_t$ is a Poincaré type metric. 
We will however exhibit in Section \ref{sec:exampledim1} examples showing that 
$\beta_t$ can have quite different asymptotics at the singularities.

\begin{proof}
Fix $\e>0$, and observe that $\# \{ m_j,  \, m_j \geq \varepsilon \}< +\infty$. 
Starting at time $\e$ and using the semi-group property, we can therefore
assume that $\mu_0$ has a finite number of atoms. 

Relabelling if necessary we assume that the sequence $j \mapsto m_j$ decreases to zero.
As it  is not necessarily strictly decreasing, we introduce
$$
t_0=m_0=\cdots=m_{i_0}>t_1=m_{1+i_0}=\cdots=m_{i_1}> \cdots > t_{\ell+1} =m_{1+i_{\ell}}=\cdots=m_{i_{\ell+1}}.
$$

 We work on the time interval  $t \in I=[t_{\ell+1},t_{\ell})$.
It follows from the last statement of Lemma \ref{lem:singdim1} that $\nu (\varphi_t, a_j)= \max(m_j-t, 0)$ 
and  $\nu(\varphi_t, x)=0$ for any $x\neq a_j$. We infer
 $$
 \f_t=\sum_{j=0}^{J_{\ell}} (m_j-t) \log |\sigma_j|^2+\p_t,
 $$
 where $\p_t$ has zero Lelong number at all points.
 We set $A_I=\{a_0,\ldots,a_{J_{\ell}} \}$.
Thus
$$
\omega_t=\sum_{j=0}^{J_{\ell}} (m_j-t) \delta_{a_j}+\beta_t,
$$
where $\beta_t$ is a K\"ahler form in the Zariski open set $X \setminus  A_I$,
by \cite[Theorem 1.1]{DNL17}.

\smallskip

The semi-group property ensures that $\omega_t$ is also the solution 
at time $t-t_{\ell+1}$ of the flow starting from $\omega_{t_{\ell+1}}$.
It follows from Proposition \ref{prop:complete} below that,
in $X \setminus  A_I$,
$$
\beta_t =e^{\dot{\varphi}_t} \omega \geq c (t-t_{\ell+1}) \omega_I,
$$
where $\omega_I$ denotes the Poincaré metric along  $A_I$.
Since the latter is complete, so is $\beta_t$.
\end{proof}

\begin{prop} \label{prop:complete}
Assume that $\f_0=\sum_{j=1}^q m_j \log |\sigma_j|^2+h_0$, where $h_0 \in \PSH(X, a\omega)$, 
for some $a\in (0,1)$. Then there is a constant $C>0$ such that for all $t>0$,
$$
	\dot{\varphi}_t \geq  \sum_{j=1}^q \left( - \log |\sigma_j|^2 - 2 \log (-\log |\sigma_j|^2) \right) +  \log t-C.
$$	
\end{prop}

\begin{proof}
We assume $q=1$ and $m_1=1$ to simplify the exposition. When $h_0$ is smooth, this inequality has been proved in any dimension in  \cite[Proposition 4.1]{DGL26}, namely
\begin{equation}
	\label{eq: completeness 1}
	\dot{\varphi}_t \geq - \log |\sigma|^2  - 2 \log (-\log |\sigma|^2)  + \log t -C, 
\end{equation}
for some constant $C$ that (a priori) depends on $h_0$.
We are going to show -in the context of Riemann surfaces- 
that $C$ can be chosen independently of $h_0$, hence \eqref{eq: completeness 1} is also valid
for non smooth $h_0$, as will follow from an approximation argument.

\smallskip

Assume first that $h_0$ is smooth and consider, fixing $\e \in (0,1/2)$,
\[
H(t,x):= \dot{\varphi}_t + (1-\varepsilon)\log |\sigma|^2  + 2 \log (-\log |\sigma|^2)  - (1+\varepsilon) \log t. 
\] 
It follows from the qualitative lower bound \eqref{eq: completeness 1}
that $H$ attains its minimum at a point $(t_0,x_0)$ with $t_0>0$, $\sigma(x_0)\neq 0$. 
Recall that $\ddot{\f}_t=\Delta_t\, \dot{\f}_t$ and 
observe that 
$$
dd^c \log |\sigma|^2=-\Theta_{1}
\; \; \text{ in }  \; \;
X \setminus (\sigma=0),
$$

while 
$$
-2dd^c \log(-\log |\sigma|^2) \geq \frac{-2 \Theta_1}{(-\log |\sigma|^2)}+\frac{2c'_0 \omega}{|\sigma|^2(-\log |\sigma|^2)^2},
$$
for some uniform constant $c'_0>0$. At the point $(t_0,x_0)$ we thus obtain
	\begin{flalign*}
			0& \geq (\partial_t -\Delta_t) H  \\
			& \geq - \frac{1+\varepsilon}{t}  + {\rm Tr}_t\left ((1-\varepsilon)\Theta_1
			-\frac{2 \Theta_1}{(-\log |\sigma|^2)}   
			+\frac{2c'_0 \omega}{|\sigma|^{2} (\log |\sigma|^2)^{2}}   \right ) \\
			&\geq - \frac{2}{t}  + \frac{2c_0}{|\sigma|^{2} (\log |\sigma|^2)^{2}} {\rm Tr}_t\left ( \omega  \right ),
	\end{flalign*}
	for some $c_0\leq c_0'$, as we can normalize the  metric so that $|\sigma |\leq 10^{-3}$.
  Using that ${\rm Tr}_t\left ( \omega  \right )=e^{-\dot{\varphi}_t}$  we infer that at $(t_0, x_0)$
\[
\dot{\varphi}_t  \geq \log t   - \log |\sigma|^2 - 2\log (-\log |\sigma|^2)+\log c_0,
\]
hence 
$$
H(t_0,x_0) \geq -\e \log t_0-\e \log |\sigma|^2 + \log c_0.
$$
%for some constant $C_1$ that depends neither on $t_0, \e$, nor on $h_0$. 
Letting $\e$ tend to zero provides the following quantitative version of \eqref{eq: completeness 1},
\begin{equation}	\label{eq: completeness 2}
	\dot{\varphi}_t \geq - \log |\sigma|^2  - 2 \log (-\log |\sigma|^2)  + \log t +\log c_0.
\end{equation}

\smallskip

We now explain how to extend this estimate to the case when $h_0$ is not necessarily smooth.
We approximate $h_0$ by a decreasing sequence $h_{0,j}$ of smooth $a\omega$-psh functions, and let
$\f_{t,j}$ denote the solution of the flow starting from
$$
\f_{0,j}=\log |\sigma|^2+h_{0,j}.
$$
It follows from the maximum principle that $\f_{t,j}$ decreases to $\f_t$
as $j$ increases to $+\infty$. Now
$$
\dot{\varphi}_{t,j} \geq - \log |\sigma|^2  - 2 \log (-\log |\sigma|^2)  + \log t +\log c_0.
$$
by \eqref{eq: completeness 2}, and $t \mapsto \f_{t,j}-n(t\log t -t)$ is concave,
hence $\dot{\varphi}_{t,j} \rightarrow \dot{\varphi}_{t}$ 
(see \cite[Theorem 1.14]{GLZ20}) and the conclusion follows.
\end{proof}

\begin{remark}
One can adapt the arguments and obtain a similar result for
the classical (untwisted) Ricci flow. The maximal existence time $T_{max}$ in this case can
be finite (if $X$ is the Riemann sphere), so one needs to compare
it with the size of the singularities of $\mu_0$ (see  \cite[Section 6]{GZ17}).

Giesen, Topping and Yin have more generally shown in \cite{GT11,Top15, TY24} that,
on any open Riemann surface $\Omega$, there exists a unique Ricci flow starting from
an arbitrary positive Radon measure with no atom, and that the resulting metric
$\beta_t$ becomes instantaneously complete.
 Theorem \ref{thm:riemsurface} is in line with these findings, and provides a more precise
 description of the metric $\beta_t$ when the measure has finite mass and the
 surface $\Omega$ is quasi-projective.
 \end{remark}

\subsection{Examples} \label{sec:exampledim1}

We show here that the complete metrics $\omega_t$
can have various asymptotics, besides the Poincaré behavior already identified in \cite{DGL26}.

We fix $\sigma \in H^0(X,L)$ a holomorphic section of an ample line bundle $L \rightarrow X$ that
vanishes at order $1$ at a point $a \in X$. We  let $\omega/2$ denote the curvature
of an hermitian metric of $L$, so that 
$
\frac{\omega}{2}+dd^c \log |\sigma|^2=\delta_a.
$
Assume that 
$$
\f_0=\log|\sigma|^2+\chi_1 \circ \log |\sigma|^2,
$$ 
where  $\chi_1$ is convex increasing, with $\chi_1'(-\infty)=0$
  and $\chi_1 \circ \log |\sigma|^2 \in \PSH(X,\frac{\omega}{2})$.
  
  \smallskip
  
We fix a second convex increasing weight $\chi_2$ with slower growth $\chi_2'' \leq \chi_1''$, $\chi_2'(-\infty)=0$,
 such that $\chi_2=\log \chi_1''$ and $\chi_2 \circ \log |\sigma|^2 \in \PSH(X,\frac{\omega}{2})$. 
 
\begin{lem} 
There exists $C>0$ such that  for all $t \in (0,1)$,
$$
-Ct +\p_t \leq \f_t  \leq \p_t+Ct,
$$
 where 
$$
\p_t=(1-t) \log |\sigma|^2+ \chi_1 \circ \log |\sigma|^2+t \chi_2 \circ \log |\sigma|^2.
$$
\end{lem}

 \begin{proof}
 Observe that $\p_t-Ct$ is a subsolution of the flow, since $\psi_0=\varphi_0$ and
 $$
 e^{\partial_t \p_t} \omega=\frac{e^{ \chi_2 \circ \log |\sigma|^2}}{|\sigma|^2} \omega
 =\frac{\chi_1'' \circ \log |\sigma|^2}{|\sigma|^2} \omega,
 $$
 while 
 \begin{eqnarray*}
 \omega+dd^c \p_t \geq \frac{\omega}{2}+dd^c \chi_1 \circ \log |\sigma|^2
 &\geq &\left(1-\chi_1' \circ \log |\sigma|^2 \right) \frac{\omega}{2} +c \frac{\chi_1'' \circ \log |\sigma|^2}{|\sigma|^2} \omega \\
 &\geq & c \frac{\chi_1'' \circ \log |\sigma|^2}{|\sigma|^2} \omega,
 \end{eqnarray*}
 where $c>0$ is so that $i \partial \sigma \wedge \overline{\partial \sigma}=c \omega$. We assume here that the hermitian metric has been so normalized that
 $\chi_1' \circ \log |\sigma|^2 \leq \frac{1}{2}$.
 
 \smallskip

 Now $\p_t+Ct$ is a supersolution since, in $X \setminus \{ a \}$,
 \begin{eqnarray*}
  \omega+dd^c \p_t & \leq & \left( \frac{1}{2}+c \frac{\chi_1'' \circ \log |\sigma|^2}{|\sigma|^2}  \right) \omega
 +t \left( \frac{1}{2}+c \frac{\chi_2'' \circ \log |\sigma|^2}{|\sigma|^2}  \right) \omega \\
& \leq & C \frac{\chi_1'' \circ \log |\sigma|^2}{|\sigma|^2}  \omega.
 \end{eqnarray*}
The conclusion follows from the comparison principle Theorem \ref{thm:pcp} applied with $f=h=0$ and $\rho=\log|\sigma|^2$.
 \end{proof}

 \begin{exa}
 Here are a few explicit examples exhibiting various asymptotics:
 %include the following weight functions defined for $x<<1$:
 \begin{itemize}
 \item $\chi_1(x)=-(-x)^{\alpha}$ with $0<\alpha<1$ and $\chi_2(x)=(2-\alpha) (-\log (-x))$, with
 $$
 \omega_t \sim \frac{\omega}{ |\sigma|^2 (-\log |\sigma|^2)^{2-\alpha}}
 $$
  \item $\chi_1(x)=\frac{x}{\log (-x)}$  and $\chi_2(x) \sim -\log (-x)-2 \log (\log (-x)) $;
  with
 $$
 \omega_t \sim \frac{\omega}{ |\sigma|^2 (-\log  |\sigma|^2) [\log (-\log  |\sigma|^2)]^2}.
 $$
%   \item $\chi_1(x)=\frac{x}{\log \circ \log (-x)}$,
%   $\chi_2(x) \sim -\log (-x)- \log \circ \log (-x) -2 \log \circ \log \circ \log (-x) $,
 %$$
% \omega_t \sim \frac{\omega}{ |\sigma|^2 (-\log  |\sigma|^2) \log (-\log  |\sigma|^2)  [\log \circ \log (-\log  |\sigma|^2)]^2}.
% $$
 \end{itemize}
 One can also consider  $\chi_1(x)=\frac{x}{\log \circ \log (-x)}$, or any other similar weights with slightly 
 smaller growth, thus providing a different asymptotics of $\omega_t$ near $a$.
 \end{exa}

% \vfill
% \pagebreak[4]

\section{Analytic singularities} \label{sec:logresolution}

 We assume in this section that $\f_0$ has analytic singularities.
We can thus find a log resolution  $\pi:Y \rightarrow X$ so that
$\pi^* \varphi_0=\sum_{j=1}^q m_j \log|\sigma_j|^2+\ell_0$,
where $m_j \in \R^+$, $\ell_0$ and $D_j=(\sigma_j=0)$ are smooth, and
 $$
 \pi^* (\omega+dd^c \f_0)=\sum_{j=1}^q m_j [D_j]+R_0,
 $$
with $D=\sum_j D_j$ log smooth  and $R_0$ a semi-positive closed form.

 \subsection{Poincar\'e current}

\subsubsection{Logarithmic singularities} 

The image by $\pi$ of some of the divisors $D_j$ are irreducible divisors in $X$ that 
correspond to divisorial singularities of $\f_0$.
We label these $D_1,\ldots,D_r$.
 
 We let $b_k \in \N$ denote the discrepancies of the resolution, defined by 
 $$
 K_Y=\pi^* K_X+ \sum_{k=r+1}^q b_k E_k,
 $$
 where $K_Y$ (resp. $K_X$) denotes the canonical divisor of $Y$ (resp. $X$).

Let $\sigma_j$  be a holomorphic defining section for ${D_j}=(\sigma_j=0)$,
 fix $h_j$ a smooth metric of 
the corresponding holomorphic line bundle,
and let $\Theta_j$ denote the curvature of $h_j$.
We set $b_j=0$ for $1 \leq j \leq r$ and let $dV_Y$ denote
the volume form on $Y$ such that
$$
\pi^*dV_X = \prod_{j=1+r}^q |\sigma_j|^{2 b_j} dV_Y
=\prod_{j=1}^q |\sigma_j|^{2 b_j} dV_Y.
$$ 
With these notations, following Section \ref{subsec_analytic sing}, we find
 $$
 R_0=\pi^*\omega-\sum_j m_j \Theta_j+dd^c \ell_0
 \; \; \text{ and } \; \;
c(\f_0)=\min_{1 \leq j \leq q} \frac{1+b_j}{m_j}.
$$

\begin{notation}  
We set $t_1(\f_0)=\min_{1 \leq j \leq q} \frac{m_j}{1+b_j}$ and
$$
\Theta:=\sum_{j=1}^q (1+b_j) \Theta_j=\sum_{j=1}^r  \Theta_j+\sum_{j=1+r}^q (1+b_j) \Theta_j.
$$
\end{notation}

\begin{lem}  
%Set $u_t:= \sum_{j=1}^q \max(m_j-(1+b_j)t,0) \log |\sigma_j|^2$.
If $0<t<t_1(\f_0)$ then
\begin{equation} \label{eq:decompo}
 \f_t \circ \pi=\sum_{j=1}^q [m_j-(1+b_j)t] \log |\sigma_j|^2+v_t,
\end{equation}
where $v_t \in \PSH(Y,t\Theta+R_0)$.
\end{lem}

\begin{proof}
Set $u_t:= \sum_{j=1}^q \max(m_j-(1+b_j)t,0) \log |\sigma_j|^2$.
It follows from Lemma \ref{lem:analytsing} that
$
\f_t \circ \pi \leq u_t+Ct
$
for some constant $C>0$.   Siu's decomposition theorem therefore ensures that
$
\pi^*\omega_t=\sum_{j=1}^q [m_j-(1+b_j)t] [D_j]+ S_t,
$
where $S_t$ is a positive closed current cohomologous to 
$$
\pi^* \omega -\sum_{j=1}^q [m_j-(1+b_j)t] \Theta_j= t \sum_{j=1}^q (1+b_j) \Theta_j+ R_0.
$$
The existence of  $v_t \in \PSH(Y,t\Theta+R_0)$ thus follows from the $\partial\overline{\partial}$-lemma.
\end{proof}

\subsubsection{Positivity of cohomology classes}

%The positivity properties of the cohomology classes $t \{\Theta \} + \{R_0\}$ play an important role in what follows.
 
 \begin{lem} \label{lem:big}
 The cohomology class $\lambda(\f_0) \{\Theta \} +\{R_0\} \in H^{1,1}(Y,\R)$ 
 is big, hence so are the classes $t \{\Theta \} + \{R_0\}$ for 
 $0 \leq t \leq \lambda(\f_0)$.
 \end{lem}

 We provide examples in  Section \ref{sec:cubics} showing  that the class
 $\lambda(\f_0) \{\Theta \} + \{R_0\}$ 
 is not necessarily nef, although it can sometimes be K\"ahler.

\begin{proof}
Observe that 
$
\pi^*\{\omega\}=\sum_{j=1}^q m_j \{\Theta_j\}+\{R_0\} \in H^{1,1}(Y,\R).
$
%Rescaling $\omega$ and $\f_0$, we can assume without loss of generality that $c(\f_0)=1$.
As recalled in Section \ref{subsec_analytic sing} we have
$m_j \leq (1+b_j) \lambda(\f_0)$ for all $j$, with equality for at least one index, hence
$$
\lambda(\f_0) \{\Theta \}+ \{R_0\}=\pi^*\{\omega\}+\sum_{j=1}^q (\lambda(\f_0)[1+b_j]-m_j) \{D_j \}
$$
is big, as the sum of a big and
a pseudoeffective class.
%$\sum_{j=1}^q (\lambda(\f_0)[1+b_j]-m_j) \{E_j \}$
%and a big class $\pi^*\{\omega\}$.
We infer that the classes
$$
t \{\Theta \} + \{R_0\}
=\frac{t}{\lambda(\f_0)} \left( \lambda(\f_0) \{\Theta \} + \{R_0\} \right)+(1-t/\lambda(\f_0)) \{R_0\}
$$
are big as well, since $\{R_0\}$ is semi-positive.
\end{proof}

\subsubsection{Poincaré potential in big classes}

%In this section we fix $Y$ a compact K\"ahler manifold, ${D}=\sum {D_j}$ a log smooth divisor on $Y$, and set 
We set $|\sigma|^2=\prod_{j=1} ^q|\sigma_j|^2$. 
 %where $\sigma_j$ are defining holomorphic sections for ${D_j}$.

\begin{theorem} \label{thm:poincaresing}
Let $\theta$ be a smooth closed differential form representing a big cohomology class.
There exists a unique function $\rho \in  {\mathcal E}^1(Y,\theta)$ such that
$$
(\theta+dd^c \rho)^n = \frac{e^{\rho}}{|\sigma|^{2}} dV_Y.
$$

Moreover $\rho$ has the same Lelong numbers as a $\theta$-psh function with minimal singularities,
and it is smooth in a dense Zariski open subset when $\{\theta\}$ is big and nef.
\end{theorem}

The class ${\mathcal E}^1(Y,\theta)$ has been introduced in \cite{GZ07,BEGZ10} (see also \cite[Chapter 10]{GZbook}).
It consists of  $\theta$-psh functions $\f$ that can be slightly unbounded,
but whose Monge-Amp\`ere measure $(\theta+dd^c \f)^n$ is well-defined, 
and has the property that $\f \in L^1((\theta+dd^c \f)^n)$.

\begin{proof}
 The construction  of $\rho$ is done in \cite[Theorem 3.2]{BG14}
by a variational approach, following \cite{BBGZ13}.
It follows from  \cite[Theorem 1.1]{DDL18} that $\rho$ has the same Lelong numbers as
$\theta$-psh functions with minimal singularities. Uniqueness follows from \cite[Theorem 3.1]{BEGZ10}
\end{proof}

The following is an important and challenging open problem:

\begin{quest}
Is $\theta+dd^c \rho$ a complete K\"ahler metric in a dense Zariski open subset ?
\end{quest}

 The answer is positive and classical if the cohomology class $\{\theta\}$ is K\"ahler.
 The regularity theory for  Monge-Amp\`ere equations in a big cohomology class
 is  however an open problem 
 (see \cite[Question 21]{DGZ16}), despite a recent  result \cite{DT24}.
If the class $\{\theta\}$ admits a Zariski decomposition, then
$\rho$ is smooth and $\theta+dd^c \rho$ is a K\"ahler form in a Zariski open set.
  Zariski decompositions always exist on surfaces, but not necessary in higher dimension
  (see  \cite{Bou04}).

\subsection{Uniform estimates}

\subsubsection{Purely divisorial case}

\begin{prop} \label{pro:decompobig}
If $\pi^*(\omega+dd^c \f_0)=\sum_{j=1}^q m_j [D_j]$  and $\{\Theta\}$ is nef and big,
then 
$$
\pi^* \omega_t=\sum_{j=1}^q (m_j -(1+b_j)t) [D_j] +t\Theta_\rho
$$
for all $t\in (0,{t_1(\f_0)})$,
where $\rho \in {\mathcal E}^1(Y,\Theta)$ satisfies $(\Theta+dd^c \rho)^n=|\sigma|^{-2}e^{\rho} dV_Y$. Equivalently, at the level of potentials we have
$$
\f_t \circ \pi =u_t +t \rho +n(t\log t-t),
$$
where $u_t:= \sum_{j=1}^q \max(m_j-(1+b_j)t,0) \log |\sigma_j|^2$.
%for all $(x,t) \in X \times (0,{t_1(\f_0)})$.
%where $\rho \in {\mathcal E}^1(Y,\Theta)$ satisfies $(\Theta+dd^c \rho)^n=|\sigma|^{-2}e^{\rho} dV_Y$.
\end{prop}

We only require here $\{\Theta\}$ nef and big to ensure that $\rho$ is smooth in a Zariski open set,
in order to apply the comparison principle (Theorem \ref{thm:pcp}).

\begin{proof}
Recall that $R_0=\pi^* \omega-\sum_j m_j \Theta_j+dd^c \ell_0$. The assumption
$\pi^*(\omega+dd^c \f_0)=\sum_{j=1}^q m_j [D_j]$ is equivalent to 
$R_0=0$,  which  forces $\pi^* \omega=\sum_j m_j \Theta_j$ and $dd^c \ell_0=0$.

The existence of $\rho$ is provided by  Theorem \ref{thm:poincaresing}.
It follows from \eqref{eq:decompo} that $\f_t \circ \pi =u_t +t  w_t$ with $w_t \in \PSH(Y,\Theta)$.
Note that $\pi^* \varphi_0=u_0$. Thus
$
\f_t \circ \pi \leq u_t +t  V_{\Theta}+C,
$
where
$$
V_{\Theta}=\sup \{ u \in \PSH(X,\Theta), \; u \leq 0 \},
$$
hence $\f_t \circ \pi$ has at least the same singularities as $u_t +t  V_{\Theta}$. 
%Note that since $\{\Theta\}$ is K\"ahler, the potential with minimal singularities is a bounded function.

\smallskip

Set $\p_t=u_t+t \rho+n(t \log t-t)$ and observe that it is smooth in a Zariski open subset 
$\Omega \subset Y\setminus D$  since $\{\Theta\}$ is nef and big. 
In $\Omega$ we obtain
$\pi^* \omega+dd^c \p_t=t(\Theta+dd^c \rho)$ hence
$$
(\pi^* \omega+dd^c \p_t)^n=t^n (\Theta+dd^c \rho)^n=t^n \frac{e^{\rho}}{|\sigma|^{2}} dV_Y,
$$
while
$$
e^{\partial_t \p} \pi^* \omega^n=t^n e^{\rho} e^{\partial_t u} \prod |\sigma_j|^{2b_j} dV_Y=t^n \frac{e^{\rho}}{|\sigma|^{2}} dV_Y,
$$
since $e^{\partial_t u}=\prod |\sigma_j|^{-2(1+b_j)}$. 
Thus $\p_t$ is a solution of the flow equation pulled back to $Y$; we infer $\p_t \leq \f_t \circ \pi$ by maximality.

It follows from \cite[Theorem 1.1]{DDP25} that $\rho \in {\mathcal E}(Y,\Theta)$ if and only if there exists
a K\"ahler form $\omega_Y$, $h \in {\mathcal E}(Y,\omega_Y)$ and $C'>0$ such that 
$$
V_{\Theta}+h \leq \rho \leq V_{\Theta}+C'.
$$
We can assume $h \leq 0$  and $0 \leq t \leq 1$ without loss of generality.
Thus
$$
\f_t \circ \pi + h \leq \f_t \circ \pi +t h \leq  u_t+t(V_{\Theta}+h) +C \leq u_t +t \rho+C \leq \p_t+C''.
$$
We can conclude by Theorem \ref{thm:pcp} (applied with $f=\sum_j \log |\sigma_j|^{2b_j}$) that $\p_t \geq \f_t \circ \pi$.
\end{proof}

 \subsubsection{General case}

 When $\pi^*(\omega+dd^c \f_0)=\sum_{j=1}^q m_j[D_j]+R_0$ is not purely divisorial, 
 the analysis is more involved. 
 We let
 $\rho \in {\mathcal E}^1(Y,\lambda(\f_0)  \Theta+R_0)$ denote the unique solution of 
 $$
 \left( \lambda(\f_0)  \Theta+ R_0+dd^c \rho \right)^n =e^{\rho} |\sigma|^{-2} dV_Y.
 $$

 \begin{prop} \label{pro:uniformbig}
Assume that $\{\lambda(\f_0) \Theta+R_0\}$ is a K\"ahler class. Then 
there exists $C_2>0$ such that  
for all $t \in  (0,{t_1(\f_0)})$,
$$
u_t+t \rho+n(t\log t-t)-C_2 \leq \f_t \circ \pi \leq   u_t+t \rho +C_2 t.
$$
\end{prop}

The lower bound always holds, but we need the assumption that 
$\{\lambda(\f_0)  \Theta+R_0\}$ is a K\"ahler class 
in order to establish the upper-bound:
it ensures that the Poincaré potential $\rho$ is smooth in  $Y\setminus D$ and that 
$\lambda(\f_0)  \Theta+ R_0+dd^c \rho$ is a K\"ahler current on $Y$
(i.e. dominates a K\"ahler form).

\begin{proof}
Recall that $R_0=\pi^* \omega-\sum_j m_j \Theta_j+dd^c \ell_0 \geq 0$. 
 Rescaling if necessary we can assume that $\lambda(\f_0) =1$.
Set $\p_t=u_t+t \rho+n(t \log t-t)+\ell_0$. This function is $\pi^* \omega$-psh and $\p_0=\pi^*\varphi_0$.
In $Y \setminus D$ we obtain
$$
\pi^* \omega+dd^c \p_t = t(\Theta+R_0+dd^c \rho)+(1-t) R_0,
$$
 hence
\begin{eqnarray*}
(\pi^* \omega+dd^c \p_t)^n &\geq& t^n (\Theta+R_0+dd^c \rho)^n= t^n \frac{e^{\rho}}{|\sigma|^{2}} dV_Y
\end{eqnarray*}
while
$
e^{\partial_t \p} \pi^* \omega^n=t^n \frac{e^{\rho}}{|\sigma|^{2}} dV_Y.
$
Thus $\p_t$ is a subsolution of the flow on $Y$, hence $\p_t \leq \f_t \circ \pi$.

\smallskip 

Since $\{\Theta+R_0\}$ is a K\"ahler class, the current 
$(\Theta+R_0+dd^c \rho)$ is a K\"ahler current. In particular
the smooth semi-positive form $R_0$ is bounded from above
by some multiple of $(\Theta+R_0+dd^c \rho)$, hence
$
\pi^* \omega+dd^c \p_t \leq  e^{C} (\Theta+R_0+dd^c \rho).
$
We infer 
\begin{eqnarray*}
(\pi^* \omega+dd^c \p_t)^n &\leq& e^{nC} (\Theta+R_0+dd^c \rho)^n 
= \frac{e^{\rho+nC}}{|\sigma|^{2}} dV_Y
\end{eqnarray*}
while
$
e^{\partial_t w} \pi^* \omega^n= \frac{e^{\rho}}{|\sigma|^{2}} dV_Y.
$
Thus $\p_t+Ct$ is a supersolution of the flow on $Y \setminus D$.
Theorem \ref{thm:pcp} can be applied since $\psi_t$ is smooth on $Y\setminus D$, $\psi_t\geq \varphi_t\circ \pi -Ct +t\rho$ with $\rho \in \mathcal{E}(Y, \Theta+R_0)$.
\end{proof}

  \subsection{${\mathcal C}^2$-estimates}\label{C2 est resolution}
 
Once the ${\mathcal C}^0$-estimate is under control,
one can argue as in the proof of the main result of \cite{DGL26}
to compare $\omega_t$ to a Poincaré metric along $D$.
 
  \begin{theorem} \label{thm:c2hyp}
  Assume that $\{\lambda(\f_0) \Theta+R_0\}$ is a K\"ahler class. Then  there exists $C>0$ such that, for all $t \in  (0,{t_1(\f_0)})$, 
	 $$
	 C^{-1}t^n \beta_P \leq \omega_t \leq C \beta_P,
	 $$
	 where $\beta_P$ is the Poincar\'e metric along $D=\sum_{j=1}^q D_j$ 
	 in the class $\{\lambda(\f_0) \Theta+R_0\}$.
\end{theorem}

  We can assume $\lambda(\varphi_0)=1$ and work on the interval $(0, 1]$.
  Then $\psi_t:=\varphi_t\circ \pi -u_t$ solves
\[
(R_0 + t\Theta + dd^c \psi_t)^n = \frac{e^{\dot{\psi_t}}}{|\sigma|^2} dV_Y.
\]
Set $\omega_t:=R_0 + t\Theta + dd^c \psi_t$, $v_t:= \psi_t-t\rho$
and observe that $\psi_0=0$.
It follows from Proposition \ref{pro:uniformbig} that
 $n(t\log t-t)-C \leq v_t \leq Ct$, $C>0$.

\subsubsection{Time derivative} 

One can show that $\dot{v}_t$ is bounded from above
by showing that
\[
\dot{\psi}_t \leq \frac{\psi_t-\psi_0}{t}+ n \leq \rho +C.
\]
%hence $\dot{v}_t\leq C$. 
One can indeed consider 
$H=t \dot{\psi}_t - (\psi_t-\psi_0) -(n+\varepsilon) t$, $\varepsilon>0$, 
and make sure (using an approximation) that the maximum of $H$ is attained at some point
$(t_0,x_0)$ with  $x_0\in Y\setminus D$. If $t_0>0$ 
we reach a contradiction as the following computation shows
\begin{flalign*}
	0 &\leq (\partial_t-\Delta_t) H 
	 = t \ddot{\psi}_t -n -\varepsilon - t \Delta_t \dot{\psi}_t + \Delta_t(\psi_t)\\
	&= t \tr_t(\Theta+dd^c \dot{\psi}_t) -n-\varepsilon- t \Delta_t \dot{\psi}_t + \tr_t(\omega_t - R_0-t\Theta) \leq -\varepsilon.
\end{flalign*}
Since $H =0$ along $(t=0)$, we obtain $H(t,x)\leq 0$ and the estimate follows. 

\smallskip

We next bound $\dot{v}_t$ from below by considering $G= \dot{v}_t+Av_t-n \log t$. 
Arguing by approximation as in \cite[Proposition 3.1]{DGL26}, we can ensure that $G$ attains its minimum at 
$(t_0,x_0)$ with $t_0>0$ and $x_0 \in Y\setminus D$. Using that 
\[
\ddot{v}_t=\ddot{\psi}_t = \tr_t (\Theta + dd^c \dot{\psi_t}) = \Delta_t\dot{v_t}+\tr_t(\Theta+dd^c \rho),
\] 
we obtain, at $(t_0,x_0)$,
\begin{flalign*}
	0& \geq (\partial_t-\Delta_t) G
	= \ddot{v}_t + A\dot{v_t} - nt^{-1} -\Delta_t \dot{v}_t - A \Delta_t v_t\\
	& = \Delta_t \dot{v_t} + \tr_t (\Theta+dd^c \rho) +A \dot{v}_t-nt^{-1}-\Delta_t \dot{v}_t - A\tr_t(R_0+t\Theta+dd^c \psi_t -tdd^c \rho -R_0-t\Theta)\\
	& \geq \tr_t (\Theta+dd^c \rho)+ A\dot{v}_t - nt^{-1}-An + tA\tr_t(R_0+\Theta+dd^c \rho)+ A(1-t)\tr_t(R_0)\\
	& \geq \tr_t(\Theta_{\rho}) + A\dot{v}_t - nt^{-1}-An + \tr_t(At\Theta_{\rho}+ (A(1-t)-1)R_0).
\end{flalign*} 
Here $\Theta_{\rho}:=\Theta+R_0+dd^c \rho$ is a K\"ahler current,
hence  $\Theta_{\rho} \geq \gamma R_0$ for $\gamma>0$ small.
Thus
\[
At\Theta_{\rho}+ (A(1-t)-1)R_0 \geq (\gamma At + A(1-t) -1)R_0\geq (A\gamma -1)R_0\geq 0, 
\]
 choosing $A\geq \gamma^{-1}$. 
Using  $\tr_\alpha (\beta)\geq n  \left(\frac{\beta^n}{\alpha^n}\right)^{1/n}$ we infer
\begin{flalign*}
	0& \geq \tr_t(\Theta_{\rho}) + A\dot{v}_t - nt^{-1}-An\geq   n e^{-\dot{v}_t/n} + A \dot{v}_t - nt^{-1}-An.
\end{flalign*} 
Since $\dot{v}_t\leq C$ we can find $C_1>0$ so large that
$
A\dot{v}_t \geq -\frac{n}{2} e^{-\dot{v}_t/n} - C_1,
$
hence  
\[
 e^{-\dot{v}_t/n} \leq C_2(t^{-1}+1),
\]
at $(t_0,x_0)$. Thus $G(t_0,x_0)\geq -C_3$ and we conclude that $\dot{v}_t\geq -C_4t+n\log t$.

\subsubsection{Laplacian estimate}

Set $\omega_t:= R_0 +t \Theta + dd^c \psi_t$ as above and $\Theta_\rho:=\Theta+R_0+dd^c \rho$. Consider the function $\alpha = t \log u - Av_t$, where $u= \tr_{\Theta_{\rho}}(\omega_t)$.  
We have just shown that
$$
n(t\log t-t)-Ct \leq v_t \leq Ct
\; \; \text{ and } \; \; 
-C +n \log t \leq \dot{v}_t \leq C.
$$

Let $-B$ be the lower bound for the holomorphic bisectional curvature of $\Theta_\rho$ in $Y\setminus D$. 
It follows from Yau's Laplacian estimate (see \cite{Yau78}, \cite[Lemma 14.5]{GZbook}) that
$$
\Delta_{t} \log \tr_{\Theta_{\rho}}( \omega_t) 
\geq -\frac{\tr_{\Theta_{\rho}} (\Ric(\omega_t))} {\tr_{\Theta_{\rho}}(\omega_t)}
-B \tr_{t}(\Theta_{\rho}).
$$
Now  $\Ric(\Theta_{\rho}) \leq C \Theta_{\rho}$ in $Y\setminus D$
and $\dot{\psi}_t=\rho + \dot{v}_t$, hence $\omega_t^n=e^{\dot{v}_t} \Theta_\rho^n$. 
It  follows therefore that on $Y\setminus D$,
\begin{flalign*}
\Delta_{t} \log \tr_{\Theta_{\rho}}( \omega_t) &\geq  \frac{\tr_{\Theta_{\rho}} (dd^c \dot{v}_t) -Cn } {u}-B \tr_{t}(\Theta_{\rho}).
\end{flalign*}

Arguing as in the proof of \cite[Lemma 3.4]{DGL26}, we can assume that $\alpha$
reaches its maximum at  $(t_0,x_0) \in (0,t_1) \times X \setminus D$.
We can moreover assume that $u(t_0,x_0) \geq 1$, hence
{\small
\begin{flalign*}
	0 &\leq (\partial_t-\Delta_t) \alpha\\ 
	&= \log u + t u^{-1} \partial_t u -A\dot{v}_t- t \Delta_t \log u +  A\Delta_tv_t\\
%	&= \log u + t u^{-1} \partial_t u -A\dot{v}_t- t \Delta_t \log u + A\tr_t (R_0+t\Theta+dd^c \psi_t-tdd^c \rho -R_0-t\Theta)\\
	&\leq \log u + t \frac{\tr_{\Theta_\rho}(\Theta+dd^c \dot{\psi}_t)}{u} -A\dot{v}_t
	- t \left\{ \frac{\Delta_{\Theta_{\rho}}\dot{v}_t}{u} -Cn - B\tr_{t}(\Theta_{\rho}) \right\}  
	+ A\tr_t (\omega_t-t\Theta_\rho- (1-t)R_0)\\
	&\leq \log u + t \frac{\tr_{\Theta_\rho}(\Theta_{\rho})}{u} -A\dot{v}_t +Cnt + tB\tr_{t}(\Theta_{\rho})  + An -At\tr_t (\Theta_\rho)\\
	&\leq \log u -A\dot{v}_t+ t(B-A)\tr_{t}(\Theta_{\rho})  + C.
	\end{flalign*}
	}
	Choosing $A=B+1$, we arrive at
\[
t \tr_{t}(\Theta_{\rho} )\leq \log u -A\dot{v}_t + C.
\]
Using the elementary inequality 
\[
\tr_{\Theta_{\rho}}(\omega_t) \leq n (\tr_t(\Theta_{\rho}))^{n-1} \frac{\omega_t^n}{\Theta_{\rho}^n}= n (\tr_t(\Theta_{\rho}))^{n-1}e^{\dot{v}_t},
\] 
we obtain
\begin{eqnarray*}
\log u 
&\leq &  (n-1) \log ( t \tr_t(\Theta_{\rho}))+ \dot{v}_t - (n-1) \log  t +\log n	\\
& \leq & (n-1) \log (\log u - A\dot{v}_t +C) + \dot{v}_t  -(n-1)\log t +\log n \\
& \leq & (n-1) \log \left(\frac{A\log u}{2} - A\dot{v}_t +C\right) + \dot{v}_t  -(n-1)\log t +\log n \\
& \leq & \frac{\log u}{2}  +C' -(n-1)\log t.
\end{eqnarray*}
We have used here $A \geq 2$, hence 
$\log u \leq \frac{A\log u}{2}$ since  $1 \leq u$, together with the
 elementary inequality $A (n-1) \log x \leq x+c(n,A)$.
Thus $ \log u\leq 2C'-2(n-1)\log t_0 $
 at $(t_0, x_0)$, which yields
 $\alpha \leq C''$, providing a uniform upper bound $u\leq C$ as $v_t$ is bounded from above by $Ct$.

The lower bound on $\dot{v}_t$ yields the reverse inequality, since
\[
\tr_{\omega_t}(\Theta_{\rho}) \leq n 
\left( \tr_{\Theta_{\rho}}(\omega_t) \right)^{n-1} \left (\frac{\Theta_{\rho}^n }{\omega_{t}^n} \right ) 
 \leq C e^{-\dot{v}_t} \leq C' t^{-n},
\]
concluding the proof of Theorem \ref{thm:c2hyp}.

\section{Singular divisors} \label{sec:singulardivisors}

We have made in  \cite{DGL26} a thorough study of the case when $\f_0$ has divisorial singularities along
 a log smooth divisor $D$. The analysis becomes more delicate when 
$D$ has singularities. We consider in this section various examples of such singular divisors.

   \subsection{Log canonical divisors} \label{sec:logcan}

We fix holomorphic line bundles $L_1,\ldots,L_q$ equipped with Hermitian metrics $h_1,\ldots,h_q$. 
For a holomorphic section $\sigma_i \in H^0(X,L_i)$, we denote by $|\sigma_i|=|\sigma_i|_{h_i}$ the norm of $\sigma$ with respect to  $h_i$. 
Rescaling $h_i$, we can assume that $|\sigma_i|\leq \frac{1}{1000}$ on $X$. 
We analyze here the case when there exists $m_i>0$ such that
$$
\f_0= \sum_{i=1}^q m_i \log |\sigma_i|_{h_i}^2.
$$
This function is $\omega$-psh if  and only if  $\sum_i m_i \Theta_{i} \leq \omega$, 
where $\Theta_{i}$ denotes the curvature of the metric $h_i$,
a condition that we implicitly assume.
We set $D=\sum_i D_i$ and $|\sigma|^2=\prod_i |\sigma_i|^2$.

\begin{defi}
The divisor $D$ is {\it log smooth} if each divisor $D_i$ is smooth and the $D_i$'s have simple normal crossings,
 i.e. $D$ is locally isomorphic to $H_1 + H_2 + ...+H_k$ where 
 $H_1,\ldots,H_k$ are the coordinate hyperplanes.  
 The divisor $D$ is {\it log canonical} if $c(\log|\sigma|^2)=1$.
\end{defi}

We focus in this section on the case when the divisor $D$ is log canonical. 
%In particular $\lambda (T_0)=1$.
 A classical example 
 %is the  ordinary double point,
%An explicit example 
is   the cone over the smooth quadric
$
D=\left\{ [z] \in \C \PP^n , \; 
\sum_{i=0}^{n-1} z_i^2=0 \right\},
$
which has an isolated singularity at the vertex, an {\it ordinary double point}.
%with $X=\C \PP^n$, $\omega=\omega_{FS}$ and $a=[0:\cdots:0:1]$.

\subsubsection{A logarithmic upper-bound}

The following result describes the divisorial asymptotic of $\f_t$ along  $D_i$,
 complementing  Lemma \ref{lem:analytsing}.

\begin{lem} \label{lem:log}
	There is a uniform constant $C>0$ such that for $0<t <1$,	
	\[
	\varphi_t \leq  \sum_{i=1}^q \max(m_i- t,0) \log |\sigma_i|^2 +Ct.
	\]
\end{lem}

This provides the right logarithmic behavior of 
the solution $\f_t$ when the divisor $D$ is {\it log canonical}.
This is no longer the case when $D$ is more singular. 
%and one then needs to  use a log resolution of the singularities to establish a refined first order asymptotic.

\begin{proof}
The proof is identical for one or several components 
%(taking convex combination of each involved quantity), 
so we only treat the case $q=1$ and $m_1=c(\log|\sigma|^2)=1$ for simplicity.

We approximate the initial potential $\varphi_0$ by $\varphi_{0,\varepsilon}= \log (|\sigma|^2+\varepsilon)$ and we let $\varphi_{t,\varepsilon}$ denote the unique smooth solution to the Monge-Amp\`ere flow with initial data $\varphi_{0,\varepsilon}$. 
Consider
$$
\p_{t,\e}=(1-t) \varphi_{0,\varepsilon} +C t
$$
and observe that 
$
dd^c \varphi_{0,\varepsilon} 
 = \frac{-|\sigma|^2}{(|\sigma|^2+\varepsilon)} \Theta_1 
 + \frac{\varepsilon }{(|\sigma|^2+\varepsilon)^2}\frac{ i}{\pi} \partial \sigma \wedge \overline{\partial \sigma}. 
$
Thus
$$
\omega+dd^c \p_{t,\e}=\left[ 1-\frac{(1-t) |\sigma|^2}{|\sigma|^2+\varepsilon} \right] \omega
+\frac{(1-t)|\sigma|^2}{(|\sigma|^2+\varepsilon)} (\omega-\Theta_1)
+ \frac{\varepsilon (1-t)}{(|\sigma|^2+\varepsilon)^2} \frac{ i}{\pi}  \partial \sigma \wedge \overline{\partial \sigma} \geq 0,
$$
and 
$
\omega+dd^c \p_{t,\e} \leq B \omega+ \frac{\varepsilon (1-t)}{(|\sigma|^2+\varepsilon)^2} \frac{ i}{\pi}  \partial \sigma \wedge \overline{\partial \sigma}
\leq B \omega+ \frac{1}{|\sigma|^2+\varepsilon} \frac{ i}{\pi}  \partial \sigma \wedge \overline{\partial \sigma}
$
for some $B>0$.
Since $i \partial \sigma \wedge \overline{\partial \sigma}$ has rank $1$ and
 $\partial_t \p=-\log (|\sigma|^2+\varepsilon) +C$, we infer
$$
(\omega+dd^c \p_{t,\e})^n \leq B^n \omega^n+ \frac{nB^{n-1}}{|\sigma|^2+\varepsilon} \frac{ i}{\pi}  \partial \sigma \wedge \overline{\partial \sigma} \wedge \omega^{n-1}
\leq  \frac{e^C}{|\sigma|^2+\varepsilon}  dV_X=e^{\partial_t \p} dV_X,
$$
if we choose $C>0$ large enough.
Since $\p_{0,\e}=\f_{0,\e}$, the classical maximum principle ensures that $\f_{t,\e} \leq \p_{t,\e}$.
The conclusion follows by letting $\e \rightarrow 0$, as $\f_{t,\e}$ decreases  to $\f_t$.
\end{proof}

\subsubsection*{An explicit formula}

Let $\rho \in {\mathcal E}^1(X,\omega)$ be the   solution of $(\omega+dd^c \rho)^n =e^{\rho} |\sigma|^{-2} dV_X$,
with $|\sigma|^2=\prod_{i=1}^q |\sigma_i|^{2}_{h_i}$,
as provided by Theorem \ref{thm:poincaresing}.
When $\omega+dd^c \f_0=\sum_{i=1}^q m_i[D_i]$ is purely divisorial, we obtain an elegant decomposition
of the metrics $\omega_t$, as shown by Proposition \ref{pro:decompobig}:
%\begin{prop} \label{pro:decompo}
if   $\omega=\sum_{i=1}^q m_i \Theta_{i}$, then for all $(x,t) \in X \times (0,\min_i m_i)$,
$$
\f_t=\sum_{i=1}^q (m_i-t) \log|\sigma_i|^2 +t \rho +n(t\log t-t).
$$
In particular  $\omega_t=t \omega_{\rho}$ in $X \setminus D$.
A basic question is therefore:

\begin{quest}
Does $\omega_{\rho}$ define a complete metric in the Zariski open set $X \setminus D$?
\end{quest}

The answer is positive when $D$ has orbifold singularities, but seems largely open in general.

\subsubsection{Uniform estimates}
 
Since the behavior of the flow is understood when $\omega=\sum_{i=1}^q m_i \Theta_i$,
we now assume  that $\omega>\sum_i m_i \Theta_{i}$ hence
$\f_0=\sum_{i=1}^q m_i \log |\sigma_i|^2_{h_i}$ is strictly $\omega$-psh.

\begin{prop} \label{pro:positivityA}
%Assume  $\omega>\sum_i m_i \Theta_{i}$.
Let $B$ be an upper bound on the
 holomorphic bisectional curvature  of $\omega$, and fix $C_1$ such that ${\rm Ric}(dV_X) \geq -C_1 \omega$. 
 Fix $A=2+2B+C_1$ and consider
$\rho_A \in {\mathcal E}^1(X,\omega)$ the unique solution to 
$$
(\omega+dd^c \rho_A)^n=e^{A \rho_A} |\sigma|^{-2} dV_X.
$$
Then $\omega+dd^c \rho_A \geq A^{-1} \omega$ is a K\"ahler current 
and $\rho_A $ is smooth in $X\setminus D$.
\end{prop}

\begin{proof}
We proceed by approximation. Fix $\e>0$ and consider  $\omega_{\e}=\omega+dd^c \rho_{\e}$,
where $\rho_{\e} \in \PSH(X,\omega) \cap {\mathcal C}^{\infty}(X)$ is the unique solution of
$\omega_{\e}^n=e^{A\rho_{\e}} (|\sigma|^2+\e^2)^{-1} dV_X$. We have
\begin{eqnarray*}
{\rm Ric }(\omega_{\e})
&=& {\rm Ric }(dV_X)-A dd^c  \rho_{\e}+dd^c \log (|\sigma|^2+\e^2)  \\
&\geq & -C_1 \omega-A \omega_{\e} + (A-1) \omega+ ( \omega+dd^c \log(|\sigma|^2+\e^2) )\\
&\geq & (A-1-C_1) \omega -A \omega_{\e} \\
&\geq &  (1+2B) \omega -A \omega_{\e}.
\end{eqnarray*}

We recall that for any $\alpha,\beta$ K\"ahler forms,  the classical Chern-Lu inequality (see  \cite[Proposition 7.2]{R14}) ensures that 
if the holomorphic bisectional curvature of $\alpha$ is bounded from above by $B$
and if ${\rm Ric}(\beta) \geq -C_1 \beta-C_2 \alpha$, then
$$
\Delta_{\beta} \log {\rm Tr}_{\beta}(\alpha) \geq -C_1-(C_2+2B) {\rm Tr}_{\beta}(\alpha).
$$
Applying this with $C_1=A$ and $C_2=-1-2B$ yields
$$
\Delta_{\omega_{\e}} \log {\rm Tr}_{\omega_{\e}}(\omega) \geq -A +{\rm Tr}_{\omega_{\e}}(\omega).
$$
 The maximum principle ensures ${\rm Tr}_{\omega_{\e}}(\omega)\leq A$,
 hence  $\omega \leq A \omega_{\e}$. 
 
 The conclusion follows now by letting $\e \rightarrow 0$.
 Indeed since $ (|\sigma|^2+\e^2)^{-1}$ increases to  $|\sigma|^{-2}$, it follows from the comparison principle 
 that $\rho_\varepsilon$ decreases to the unique solution $\tilde{\rho} \in {\mathcal E}^1(X,\omega)$
 of $(\omega+dd^c \tilde{\rho})^n=e^{A \tilde{\rho}} |\sigma|^{-2} dV_X$, i.e. $\tilde{\rho}=\rho_A$. 
 Thus $\omega_{\e}$ weakly converges to  $\omega+dd^c \rho_A \geq A^{-1} \omega$.
 The regularity in $X\setminus D$ follows from \cite[Theorem 4.6]{BG14}.
\end{proof}

 \begin{prop} \label{prop: sing canoniques}
 Assume  $\omega \geq \delta_0 \omega+\sum_{i=1}^q m_i \Theta_{i}$ 
 and set $\e_0=\delta_0/A$, where $A$ denotes the constant from Proposition \ref{pro:positivityA}.
Then for 
$0 < t < \e_0$, one has
$$
n (t \log t -t) \leq \f_t - u_t- At \rho_A \leq C_A t,
$$
where $u_t=\sum_{i=1}^q \max(m_i-t,0) \log|\sigma_i|^2$.
\end{prop}

\begin{proof}
 To simplify we again only treat the case $q=m_1=1$.
Recall from Lemma \ref{lem:log} that there is a uniform constant $C>0$ such that for $0<t <1$,
		\[
	\varphi_t \leq (1- t) \log |\sigma|^2 +Ct.
	\]

The assumption $\f_0=\log |\sigma|^2$ strictly $\omega$-psh
ensures the existence of $\delta_0 \in (0,1)$ such that 
$\omega-\Theta_1 \geq \delta_0 \omega$.
We infer that  $\psi_t=u_t+ At \rho_A$ is $\omega$-psh for $0 < t < \e_0:=\delta_0/A$,
with
$$
\omega+dd^c \psi_t \geq \delta_0 \omega+ At dd^c  \rho_A
\geq At \left( \omega+dd^c \rho_A \right),
$$
hence
$$
(\omega+dd^c \psi_t )^n \geq A^nt^n \left( \omega+dd^c \rho_A \right)^n
=e^{\partial_t \p_t+n \log A+n \log t} dV_X.
$$
Moreover $ \psi_0=\varphi_0$. Thus $\p_t+t n \log A+n (t\log t-t) $ is a subsolution of the flow, i.e. $\p_t+t n \log A+n (t\log t-t) \leq \f_t$.

\smallskip

Conversely we obtain, outside $D$,
	\[
\omega+	dd^c \psi_t =\omega-(1-t)\Theta_1  +  At dd^c \rho_A
\leq C (\omega+dd^c \rho_A),
	\]
hence
$
		(\omega+dd^c \psi_t)^n 
\leq  C^n (\omega+dd^c \rho_A)^n =C^n e^{\partial_t \psi_t} dV_X.
$
It  thus follows  from Theorem \ref{thm:pcp} that $\f_t \leq \p_t+C't$. 
 The latter can be applied since
 $\psi$ is smooth in $(0, 1)\times X\setminus D$ and
  $\psi_t\geq \varphi_t+At\rho_A+Ct$,  as follows from Lemma \ref{lem:log}.
\end{proof}

\subsection{Cone singularities} \label{sec:cone}

We analyze here the case of a cone singularity, i.e. assume that
\begin{itemize}
\item $D=(\sigma=0)$ is irreducible and smooth but at one point $a \in X$, and $\f_0=\log |\sigma|^2$;
\item there exists a neighborhood of $a$ in $X$ biholomorphic to a neighborhood of 
the origin $U$ in $\C^n$, with 
$D=\{ z \in U, P(z)=0 \}$, where $P$ is a homogeneous polynomial of degree $p$
such that $H =\{[z] \in \C\PP^{n-1}, P(z)=0 \}$ is a smooth hypersurface of $\PP^{n-1}$.
\end{itemize}

When $p \leq n$ the divisor $D$ is log canonical and the asymptotic behavior of the flow has been studied in 
Section \ref{sec:logcan}; we thus assume that  $p \geq n+1$. We  also assume that $\Theta_D +(p-n)dd^c L > 0$ is a K\"ahler current, where $L$ is a smooth function on $X\setminus \{a\}$ that is given near $a$ by $L(z)=\log |z|^2$. This assumption is verified when $X=\mathbb P^n$.

\subsubsection{Log resolution}

\begin{lem} \label{lem:singinitialecone}
Assume $D$ and $\f_0$ are as above with $p>n$.
Let $\pi: Y \rightarrow X$ be the blow up of $X$ at point $a$ and set $E=\pi^{-1}(a)$.
Then  
\begin{itemize}
\item $\pi^*D=\tilde{D}+pE$ is log smooth hence $\pi$ is a log resolution of $(X,D)$;
\item $K_Y=\pi^*K_X+(n-1)E$ hence $D$ is not log canonical with $c(\f_0)=\frac{n}{p}<1$; 
\item the class $\{\Theta\}:=\{\Theta_{\tilde{D}}+n \Theta_E\}$ is K\"ahler.
\end{itemize}
\end{lem}

\begin{proof}
Let $\pi:Y \rightarrow X$ denote the blow up of $X$ at the point $a$. We obtain
$$
\pi^*D=\tilde{D}+pE
$$
where $\tilde{D}$, the strict transform of $D$  is a smooth divisor locally isomorphic
to the product of $H$ and the unit disk $\D \subset \C$.
Since $E$ and $\tilde{D}$  have simple normal crossings, 
it follows that $\pi$ is a log resolution of $(X,D)$.

Using the notations from Section \ref{sec:logresolution} we have $r=1,q=2$,
$D_1=\tilde{D}$, $D_2=E$ and $m_1=1$, $m_2=p$, $b_1=0$, $b_2=n-1$, hence
$$
c(\f_0)=\min_j \frac{1+b_j}{m_j}=\min \left( 1, \frac{n}{p} \right)=\frac{n}{p}<1
$$
since $p>n$. 

The positive closed current $\pi^*(\Theta_D + (p-n) dd^c L)- (p-n)[E] $
%, defined on $Y\setminus \pi^* D$,  extends as
defines a K\"ahler form on $Y$, which is  cohomologous to 
\[
\pi^* \Theta_D - (p-n) \Theta_E = \Theta_{\tilde D} + p\Theta_E  - (p-n) \Theta_E= \Theta_{\tilde D} + n \Theta_E,
\]
 therefore $\{\Theta\}$ is a K\"ahler class on $Y$. 
\end{proof}

\subsubsection{Uniform estimates}

We let $dV_Y$ denote the smooth volume form on $Y$ defined by 
$\pi^* \omega^n=|\sigma_E|^{2(n-1)} dV_Y$.
Invoking Theorem \ref{thm:poincaresing},
we consider $\rho \in {\mathcal E}^1(Y,\Theta)$ the unique solution of the equation
$$
(\Theta+dd^c \rho)^n=\frac{e^{\rho}}{|\sigma_{\tilde{D}}|^2 |\sigma_E|^2} dV_Y.
$$
Since $\{\Theta\}$ is a K\"ahler class and the divisor $\tilde{D}+E$ is log smooth, the metric
$\Theta+dd^c \rho$ is a Poincaré type metric in $X \setminus \left(\tilde{D}+E \right)$.

%To slightly simplify the analysis, we assume that $T_0=[D]$ is the current of integration along the divisor $D$.
For $t \in (0,1)$ 
the order zero asymptotic of the solution $\f_t$ is provided by Proposition \ref{pro:uniformbig}:
$$
\f_t \circ \pi=(1-t) \log |\sigma_{\tilde{D}}|^2+(p-nt) \log |\sigma_E|^2+t w_t,
$$
where $w_t \in \PSH(Y,\Theta)$ is such that 
$$
\rho+ n(\log t -1) \leq w_t \leq \rho+C,
$$
for some constant $C>0$.
Note moreover that $w_t=\rho+ n(\log t -1)$ when $\omega+dd^c \f_0=[D]$ is purely divisorial, 
as follows from Proposition \ref{pro:decompobig}.

 We can push  this asymptotic down to $X$, obtaining
 $$
 \f_t =(1-t) \log |\sigma_{{D}}|^2 +t v_t,
 $$
where $v_t \in \PSH(X,\omega)$ has a homogeneous Lelong number of order $(p-n)$ at the point $a$.
More precisely $w_t=v_t \circ \pi -(p-n) \log |\sigma_E|^2$
and $v_t$ is comparable to $\hat{\rho} \in \PSH(X,\omega)$, defined
by $\rho= \hat{\rho}\circ \pi -(p-n) \log |\sigma_E|^2$.
Thus the flow progressively replaces a logarithmic singularity along $D$ by
$(1-t)[D]$ together with a logarithmic singularity at the point $a$ (of multiplicity $t(p-n)$),
building  a complete metric in $X \setminus D$,
as follows from Theorem \ref{thm:c2hyp}.

\smallskip

It remains to analyze  the case when $t \in (1,\frac{p}{n})$ (the metrics $\omega_t$ are K\"ahler forms on the whole of $X$
for $t > \lambda(\f_0)=\frac{p}{n}$, as follows from \cite[Theorem 2.4]{DGL26}). 
By the semi-group property, this boils down to understand the behavior of the flow when starting from an initial
datum $\f_0$ which has a homogeneous Lelong number  at the point $a$,
and an extra mild singularity of Poincaré type, modeled on
$$
\hat{\rho}=(p-n) \log d(\cdot,a)^2+ \pi_* \rho,
$$
where $\rho$ is a Poincar\'e type potential on a resolution of singularities.
Understanding these ${\mathcal C}^0$-estimates is the purpose of Section \ref{sec:isolated}.

 \subsection{Plane cubics} \label{sec:cubics}

 Algebraic curves  $D \subset \PP^2$ of low degree $p \in \N^*$ provide a natural source of examples.
When $1 \leq p \leq 2$ the divisor $D$ is necessarily log smooth, a case that has been
fully understood in \cite{DGL26}. When $p=3$, then $D$ is either log canonical, or 
falls in one of the following three cases of interest:

\begin{itemize}
\item three  lines $D=D_1+D_2+D_3$ intersecting at a single point $a$; one obtains a log resolution by blowing up
the point $a$ with $c(D)=\frac{2}{3}$, and the analysis is very similar to the case treated in Section \ref{sec:cone};
\item $D=D_1+D_2$, where $D_1$ is a line tangent to a conic $D_2$ at some point $a$, 
%one obtains $c(D)=\frac{3}{4}$ and 
the analysis requires two blow ups and will be performed here below;
\item $D \sim \{x^2=y^3 \}$ is isomorphic to the cuspidal cubic, one obtains $c(D)=\frac{5}{6}$ and 
the analysis requires three blow ups, in a similar vein as the previous case.
\end{itemize}

From now on $X=\PP^2$ and 
we consider $\f_0=\log|\sigma|^2$ with
$D=(\sigma=0)=D_1+D_2 \subset \PP^2$, 
where  $D_1$ is a line and $D_2$ is a conic tangent to $D_1$ at some point $a$.
We let $\pi_1:Y_1 \rightarrow X$ denote the blow-up of $X$ at $a$. Observe that 
the strict transforms $D_1',D_2'$ meet the exceptional divisor $E_1'$ at the same point $b$. We let 
$\pi_2:Y \rightarrow Y_1$ denote the blow up of $Y_1$ at $b$.

\smallskip

Set $\pi=\pi_2 \circ \pi_1: Y \rightarrow X$, $E_2=\pi_2^{-1}(b)$, and let 
$\tilde{D_1}, \tilde{D_2}$ denote the strict transform of $D_1,D_2$ by $\pi$,
and $E_1$ denote the strict transform of $E'_1$ by $\pi_2$.
Then  
\begin{itemize}
\item $\pi^*D=\tilde{D_1}+\tilde{D_2}+2E_1+4E_2$ is log smooth hence $\pi$ is a log resolution of $(X,D)$;
\item $K_Y=\pi^*K_X+E_1+2E_2$ hence $D$ is not log canonical with $c(\f_0)=\frac{3}{4}<1$.
\end{itemize}

Recall that for each pseudoeffective $\R$-divisor $\hat{D}$ on a compact surface, there exists a unique
(Zariski) decomposition $\hat{D}=P+N$, where
\begin{itemize}
\item   $N=\sum_i a_i N_i$ is an effective $\R$-divisor with irreducible components $N_i$ and $a_i>0$;
\item $P$ is a nef $\R$-divisor with $ P \cdot  N_{i}=0$ for all $i$;
\item $(N_i \cdot N_j)_{i,j}$ is negative definite.
\end{itemize}
Moreover $\hat{D}$ is K\"ahler if and only if $N=0$ and $P$ is K\"ahler.

\smallskip

It follows from Lemma \ref{lem:big} that the cohomology class 
$$
\{\hat{D} \}=\{\Theta\}=\{\Theta_{\tilde{D_1}}\}+\{\Theta_{\tilde{D_2}}\}+2 \{\Theta_{E_1}\}+3\{\Theta_{E_2}\}
$$
is big, hence pseudoeffective. 
It is however not K\"ahler in this example:

\begin{lem} \label{lem:conicline}
The Zariski decomposition of  $\{\Theta\}=\{\Theta_{\tilde{D_1}}\}+\{\Theta_{\tilde{D_2}}\}+2 \{\Theta_{E_1}\}+3\{\Theta_{E_2}\}$ is 
$$
P=\pi^* \{\Theta_D \}-\frac{1}{2} \{\Theta_{E_1}\}-\{\Theta_{E_2}\}=\{\Theta\}-\frac{1}{2} \{\Theta_{E_1}\}
\; \text{ and } \;
N=\frac{1}{2} \{\Theta_{E_1}\}.
$$
\end{lem}

\begin{proof}
Observe that  $D_1^2=1$ and $D_2^2=4$. Every time we blow up the intersection number decreases by one,
so we obtain $\tilde{D_1}^2=-1$ and $\tilde{D_2}^2=2$. 
Similarly,  $(E_1')^2=-1$ and $E_1^2=-2.$ Finally observe that $E_2^2=-1$, $\tilde{D}_1\cdot E_2=\tilde{D}_2\cdot E_2=E_1\cdot E_2=1$ and $\tilde{D}_1\cdot \tilde{D}_2= \tilde{D}_1\cdot E_1= \tilde{D}_2\cdot E_1=0$.
We let the reader check that 
$$
P\cdot \tilde{D}_1=2, \;  P\cdot \tilde{D}_2=5, \;  P\cdot E_1=0, \text{ and }  P\cdot E_2= 1/2.
$$
Since $P$ is an effective divisor, its intersection with any curve not contained in its support is non-negative, being equal to the number of intersection points counted with multiplicity. It follows that $P$ is nef. 
Moreover $N$ is effective, irreducible,  with $N^2 =-1/2<0$ and $P\cdot N=0$. 
The result follows by uniqueness of the Zariski decomposition. 
\end{proof}

\begin{remark}
Consider $[D]=[D_1]+[D_2]$. The class $\{\Theta\}$ is not nef so we can not apply 
Proposition \ref{pro:decompobig}. 
We nevertheless expect that the same conclusion should hold.
This would then lead, for $0<t<1$, to the decomposition
  $$
  \pi^*\omega_t=(1-t) \left( [\tilde{D_1}]+[\tilde{D_2}]+2 [E_1] \right) +(4-3t) [E_2]+\frac{t}{2} [E_1]+t (\beta+dd^c \rho),
  $$
 where $\beta=\pi^* \Theta_D -\frac{1}{2} \Theta_{E_1}-\Theta_{E_2}$ represents the nef and big class $P$,
and $\rho \in {\mathcal E}^1(Y,\beta)$ solves
 $$
 (\beta+dd^c \rho)^2=\frac{e^{\rho}}{|\sigma_{\tilde{D}_1}|^2 |\sigma_{\tilde{D}_2}|^2 |\sigma_{E_1}|^{3/2} |\sigma_{E_2}|^2  } dV_Y.
 $$
In other words 
 $
 \pi^* \omega_t=(1-t) \pi^*[D]+t [E_2]+\frac{t}{2} [E_1]+t \beta_{\rho}.
 $
\end{remark}

%\vfill
%\pagebreak[4]

\section{Isolated singularities}  \label{sec:isolated}

We   consider in this section the case when $\f_0$ has isolated logarithmic singularities.
We obtain a neat understanding of the asymptotics when the singularities are homogeneous,
and analyze another family of toric singularities. We eventually develop the first steps
of the general situation by combining singularities with different dimensions.

\subsection{Homogeneous isolated singularities} \label{sec:isolatedhomogeneous}

We consider here an initial potential $\varphi_0$ which  has an isolated {\it homogeneous logarithmic singularity}
at some point $a \in X$.
In a local chart where $a=0 \in \C^n$, this means that 
$\f_0=\gamma \log |z|^2+r_0(z)$, where $r_0$ is smooth  and $\gamma>0$.
By rescaling we reduce to the case $\gamma=1$,
which is equivalent to  $c(\f_0)=n$ and $\la(\f_0)=1/n$.

\subsubsection{${\mathcal C}^0$-estimates}

\begin{lem} \label{lem:log1}
There exists $C_1>0$ such that  for all $(t,x) \in [0,1/n] \times X$,
$$
\f_t(x) \leq (1-nt) \f_0(x)+C_1t.
$$
\end{lem}

\begin{proof}
Set $\f_{0,\e}=\chi\circ  \log(|z|^2+\e^2)+r_0$, where $\chi$ is a cut-off function  such that
$\chi \equiv 1$ near $a=0$. The functions $\f_{0,\e}$ are smooth $\omega$-psh functions that decrease
towards $\f$ as $\e \searrow 0$. Observing that
$$
dd^c \log(|z|^2+\e^2) \leq \frac{dd^c |z|^2}{|z|^2+\e^2},
$$
we can use the binomial expansion to conclude that
$$
(\omega+dd^c \f_{0,\e})^n \leq \frac{e^C}{(|z|^2+\e^2)^n} dV_X.
$$
for some uniform  $C>0$.
The smooth $\omega$-psh functions $\p_{t, \varepsilon}=(1-nt) \f_{0,\e}+Ct$ thus satisfy
$$
(\omega+dd^c \p_{t,\e})^n \leq (\omega+dd^c \f_{0,\e})^n \leq e^{\partial_t \p_{t,\e}} dV_X,
$$
i.e. they are super-solutions for the (smooth) flow starting at $\f_{0,\e}$. 
The classical maximum principle ensures that $\f_{t,\e} \leq \p_{t,\e}$,
and the conclusion follows by letting $\e$ tend to zero.
\end{proof}

Translating $\f_0$ if necessary we can assume that  $-2\log(-\f_0)$ is $\omega$-plurisubharmonic, 
%It follows that the function 
hence
$$
x \mapsto \psi_t(x)=(1-nt) \f_0(x)- 2t \log (-\f_0(x))+r_0(x)
$$
is an $\omega$-psh function.

\begin{prop}  \label{pro:c0isolatedhomog}
There exists $C_2>0$ such that 
for all $(t,x) \in (0,1/n) \times X$,
$$
-C_2t +h(t)  \leq \varphi_t- \psi_t \leq C_2t,
$$
where $h(t)=(t \log t-t)-\frac{n-1}{n} \left((1-nt)\log (1-nt) +nt \right)$.
\end{prop}

\begin{proof}
 Since the singularity is isolated, we perform the estimates in a local chart $U$ near  
$a=0 \in U \subset \C^n$.
 Set $L(z):= \log |z|^2$. Then  
$$
\omega+dd^c \psi_t= (1-nt) dd^c L +\frac{2t}{(-L)} dd^c L+ \frac{2t}{(-L)^2} \, dL \wedge d^c L+\eta,
$$
where $\eta$ is a smooth positive $(1,1)$-form. Recall that $(dd^c L)^n$ is a Dirac mass at $0$ and $dL \wedge d^c L$ has rank $1$.
 In the binomial expansion of $(\omega+dd^c \psi_t)^n$, 
 the dominant term outside of $0$ is $2n t (1-nt)^{n-1} (-L)^{-2}(dd^c L)^{n-1} \wedge dL \wedge d^c L$.
 Thus for all $0 <t <1/n$ we obtain
$$
(\omega+dd^c \psi_t)^n \leq c_n \frac{(dd^c L)^{n-1} \wedge dL \wedge d^c L}{(-L)^2} \leq 
e^C \frac{ dV}{|z|^{2n}(\log |z|^2)^2} \leq e^{\partial_t (\p_t+C' t)} dV_X.
$$
Therefore $\psi_t+C' t $ is a supersolution of the flow in $X \setminus \{a\}$.
Theorem \ref{thm:pcp} ensures that $\f_t \leq \p_t+C't$. Observe that Theorem \ref{thm:pcp} can be applied since $\psi_t \geq \log |z|^2$ and 
by Lemma \ref{lem:log1} we have that $\psi_t +C't\geq \varphi_t + t (-2\log (-\varphi_0))$, where $t (-2\log (-\varphi_0))\in \mathcal{E}(X, \omega)$.

\smallskip

The same computation shows that 
$$ 
(\omega+dd^c {\psi}_t)^n \geq   2n t(1-nt)^{n-1} \frac{(dd^c L)^{n-1} \wedge dL \wedge d^c L}{(-L)^2}  
\geq   c t(1-nt)^{n-1}  e^{\partial_t \p_t} dV_X.
$$
We infer that  ${\psi}_t-Ct +h(t)$ is a subsolution for some $C>0$, 
where $h$ is the   ${\mathcal C}^1$-smooth function such that $h(0)=0$ and
$e^{h'(t)}=t(1-nt)^{n-1}$.
Thus  ${\psi}_t-Ct +h(t) \leq \f_t$ as claimed.
\end{proof}

\subsubsection{${\mathcal C}^{1,2}$-estimates}

 Fix $\chi$ a cut-off function in a local chart near $a=0$ such that
 \begin{itemize}
 \item $\chi$ has support in the ball $\mathbb{B}(1/2)$ of radius $1/2$;
 \item $0 \leq \chi \leq 1$ on $X$ and $\chi \equiv 1$ in $\mathbb{B}(1/4)$.
 \end{itemize} 
 For $0<\delta$ small enough, the function 
 $\rho_a=\delta \chi \left[ \log |z|^2 -2\log (-\log |z|^2) \right]$ is  $\omega$-psh.
 
 \begin{defi}
 We set
 $ \beta_a:=\omega+dd^c \rho_a$.
\end{defi} 

The following properties are classical (see \cite{MY83}):
\begin{itemize}
\item $\beta_a$ is a K\"ahler form in $X \setminus \{a\}$ and a K\"ahler current on the whole $X$;
\item the metric $\beta_a$ is complete in $X \setminus \{a\}$.

\item the holomorphic bisectional curvature of $\beta_a$ is bounded.
\end{itemize}

We  now show that the metrics $\omega_t$ are comparable to the model metric $\beta_a$.

\begin{theorem} \label{thm:isolatedhomog}
Set $v_t:=\varphi_t-\psi_t$. There exists $C>0$ such that  for all $t \in (0,1/n)$, 
$$
n\log t -C \leq \dot{v}_t \leq C
\; \; \text{ and } \; \; 
e^{-C/t} \beta_a \leq \omega_{t} \leq e^{C/t} \beta_a.
$$

In particular $\omega_{t}$ is a complete metric in $X\setminus \{a\}$.
\end{theorem}

\begin{proof}
{\it First order estimate.}
Consider
\[
\rho:=  \varphi_0 + (-2 \log (-\varphi_0)).
\]
Observe that it is $\omega$-psh (up to adding a constant to $\f_0$) 
and $\omega_\rho$ is uniformly comparable with $\beta_a$. 
In particular $\omega_{\rho} \geq c \omega_X$, for some constant $c>0$. 
Fix $\eta>0$ small. For $0\leq t\leq \frac{1}{n}-2\eta$ we 
obtain  $(1-nt)\geq \eta t$ and a direct computation yields
$$
	\omega+dd^c \psi_t \geq \eta t \omega_\rho,
\; \; 	\text{ and } \; \;
	\omega + dd^c (\psi_t + \eta\dot{\psi}_t) \geq \eta \omega_\rho. 
$$

We now consider 
$
G=\dot{v}_t + A v_t -H(t),
$
where $A>1/ \eta$ is a large constant (to be chosen later), $v_t= \varphi_t-\psi_t$
and $H(t)=n\log t$. 
Approximating $\varphi_0$ by smooth decreasing approximants $\varphi_{0,j}$, 
we can ensure that $G$ attains its minimum at some $(t_0,x_0) \in (0,T_0]\times X\setminus \{a\}$, 
where $T_0<1/n-2\eta$. 
Since $ \ddot{v}_t = \ddot{\varphi}_t =\Delta_t \dot{\varphi}_t$,
standard computations yield
\begin{flalign*}
	0 & \geq (\partial_t-\Delta) G = \ddot{v}_t + A\dot{v}_t-H'(t) - \Delta_t(\dot{v}_t)-A\Delta_t(v_t) -A { \tr}_t(\omega_{\varphi_t}-\omega_{\psi_t}) \\
	& \geq   A\dot{v}_t -H'(t) + \tr_t(A\omega+dd^c (A\psi_t+ \dot{\psi}_t))-An\\
	& \geq A\dot{v}_t -H'(t)-An + \tr_t(\omega_{\rho}). 
\end{flalign*}

The computation from Proposition \ref{pro:c0isolatedhomog} ensures that $\omega_{\rho}^n$ is comparable to $e^{\dot{\psi}_t} dV_X$. Using $\omega_{t}^n=e^{\dot{\varphi}_t} dV_X= e^{\dot{v}_t+\dot{\psi}_t} dV_X$ 
and the elementary inequality 
\[
\tr_t(\omega_{\rho}) \geq n \left (\omega_{\rho}^n/\omega_{t}^n \right )^{1/n} \geq e^{-C-\dot{v}_t/n}
\]
together with $Ax+e^{-x/n} \geq ce^{-x/n} $ for some $c>0$, we obtain
$
0 \geq ce^{-\dot{v}_t/n} -H'(t) -C,
$
which yields $\dot{v}_t\geq -n \log (H'(t)+C)+n\log c$.  Thus 
$$
G(t_0,x_0) \geq -n \log (H'(t_0) +C) +n\log c+Av_{t_0}(x_0)-H(t_0).
$$ 
Since $v$ is bounded (Proposition \ref{pro:c0isolatedhomog}) and $H(t)=n\log t$, 
we get $G(t_0,x_0)\geq -C$ hence
\[
\dot{v}_t \geq n\log t- C. 
\] 

For the upper bound of $\dot{v}_t$ we use the quasi concavity of $\varphi_t$ \cite[Proposition  1.15]{DGL26}, together with the $C^0$-estimate (Proposition \ref{pro:c0isolatedhomog}) to get
\begin{eqnarray*}
\dot{v}_t= \dot{\varphi_t} -\dot{\psi}_t 
&\leq& \frac{\varphi_t-\varphi_0 }{t} +n +n\varphi_0+2\log (-\varphi_0)\\
&\leq& \frac{\psi_t-\varphi_0 }{t} +C_2+n+n\varphi_0+2\log (-\varphi_0)=C_2+n.
\end{eqnarray*}

\medskip

\noindent {\it Second order estimate.}
We now set $u:= \tr_{\omega_{\rho}}(\omega_{t})$ and consider
\[
\alpha = t \log u - Av_t. 
\]
We can again ensure that the maximum of $\alpha$ is attained at  $(t_0,x_0)\in (0,T_0]\times X\setminus\{a\}$. 
At this point, following the same arguments as in the proof of \cite[Lemma 3.4]{DGL26}, 
the maximum principle yields 
\begin{flalign*}
0 \leq (\partial_t -\Delta_t) \alpha 
&= \log u + t\frac{\partial_t u}{u} -A\dot{v}_t+A \tr_t(\omega_{t}-\omega_{\psi_t}) -t \Delta_t \log u\\
	&\leq  \log u + t\frac{\partial_t u}{u} -A\dot{v}_t+An-A\eta t \tr_t(\omega_{\rho}) -t \Delta_t \log u\\
	&\leq    \log u -A\dot{v}_t+An-A\eta t \tr_t(\omega_{\rho})+t(B+C_1)\tr_{t}(\omega_{\rho})\\
	& \leq \log u + C - A \dot{v}_t -t(A\eta-B-C_1) \tr_t(\omega_{\rho})\\
	& \leq \log u + C - A \dot{v}_t -t\tr_t(\omega_{\rho}),
\end{flalign*}
where $B$ is a lower bound for the holomorphic bisectional curvature of $\omega_{\rho}$, $C_1$ is so that 
${\rm tr}_{\omega_{\rho}}(\Ric (dV_X))\leq C_1$ and we choose $A=1+\frac{B+C_1}{\eta}$. 

Using that $(n-1)\log x \leq x/2 +C$ for some large $C>0$, we obtain that at $(t_0, x_0)$,
\[
(n-1)\log t + (n-1)\log \tr_t(\omega_{\rho}) \leq (n-1)\log ( \log u +C -A\dot{v}_t)\leq \frac{\log u}{2} -\dot{v}_t +C. 
\] 
The elementary inequality 
\[
\tr_{\omega_{\rho}} (\omega_{t}) \leq n \tr_{t}(\omega_{\rho})^{n-1}\frac{\omega_t^n}{\omega_{\rho}^n} \leq C\tr_t(\omega_{\rho})^{n-1} e^{\dot{v}_t}, 
\]
can now be used to obtain, at $(t_0, x_0)$,
\[
\log u \leq C + (n-1)\log \tr_t(\omega_{\rho}) + \dot{v}_t \leq \frac{\log u}{2} -(n-1)\log t +C. 
\]
Thus $\log u \leq C- (n-1)\log t_0$ and $\alpha(t_0,x_0)\leq C$ 
(recall that $v_t$ is bounded by Proposition \ref{pro:c0isolatedhomog}). We conclude that
$
\log \tr_{\omega_{\rho}}(\omega_{t}) \leq C/t.
$
%or equivalently $\omega_{\varphi_t}\leq e^{C/t} \omega_\rho$. 

The lower bound on $\dot{v}_t$ yields the reverse inequality, since
\[
\tr_{\omega_{t}} (\omega_{\rho}) 
\leq n \tr_{\omega_{\rho}}(\omega_{t})^{n-1}\frac{\omega_{\rho}^n}{\omega_{t}^n} 
\leq C  \tr_{\omega_{\rho}}(\omega_{t})^{n-1}  e^{-\dot{v}_t}.
\]
\end{proof}

\subsection{Toric isolated singularities}

The situation is more intricate for singularities that are not homogeneous.
We analyze in this section a family of isolated toric singularities.

\subsubsection{Monge-Amp\`ere potential at the critical exponent}

We first establish the following result of independent interest.

\begin{lem} \label{lem:closeopen}
Fix $\varphi_0 \in \PSH(X,\omega)$ and set  $c=c(\f_0)$. 
There exists a unique function $\rho\in \mathcal{E}(X,\omega)$ such that 
	\begin{equation} \label{eq_crit}
	(\omega+dd^c \rho)^n = e^{c \rho -c\varphi_0} dV_X.
	\end{equation}
\end{lem}

Let us recall that $e^{-c\f_0}$ is not integrable at the critical exponent $c=c(\f_0)$, as follows from
the solution of the openness conjecture 
%of Demailly-Kollar \cite{DK01} 
by  Berndtsson \cite{Bo15} (see also \cite{GZh15,H14}).
The above result shows that one can recover integrability at the critical exponent by 
twisting by a function with mild singularities.

\begin{proof}
Rescaling we can assume that $c=1$.
We are going to prove the following slightly more general result:
if $\psi^{\pm}$ are quasi-psh functions such that
$e^{\psi^+-\psi^-}\in L^{1-\varepsilon}$ for all $\varepsilon\in (0,1)$,
then there exists a unique  $u\in \mathcal{E}(X,\omega)$ such that 
	\[
	(\omega+dd^c u)^n = e^{u+\psi^+-\psi^-} dV_X.
	\] 
The lemma follows by taking $\psi^+=0$ and $\psi^-=\f_0$.

 We approximate $\psi^-$ by a decreasing sequence $\psi^-_j$ of bounded $\omega$-psh functions.
 It follows from \cite[Corollary 11.9 and Theorem 12.1]{GZbook} that
there exists a unique function $u_j \in \PSH(X,\omega) \cap L^{\infty}(X)$ that solves
\[
(\omega+dd^c u_j)^n = e^{u_j+\psi^+-\psi^-_j}dV_X. 
\]

The comparison principle ensures that $j \mapsto u_j$ is decreasing.
Fix $b>1$. By \cite[Lemma 4.4]{DDL21} the envelope $v_j:=P_{\omega}(bu_j-(b-1)\psi_j^-)$ satisfies
\[
(\omega+dd^c v_j)^n \leq {\bf 1}_{\{v_j=bu_j-(b-1)\psi_j^- \}}b^n (\omega+dd^c u_j)^n \leq b^ne^{b^{-1}v_j+\psi^+-b^{-1}\psi_j^-} dV_X. 
\]
 
The $L^p$-norm of the densities $e^{\psi^+-b^{-1}\psi_j^-}$ are uniformly bounded in $j$, whenever
$p<b$. Indeed,
$$
\int_X e^{p\psi^+-\frac{p}{b}\psi_j^-} dV_X
\leq C \int_X e^{-\frac{p}{b}\psi_j^-} dV_X \leq  C \int_X e^{-\frac{p}{b}\psi^-} dV_X.
$$
It follows therefore from \cite[Theorem 12.1]{GZbook} that $v_j\geq -C(b)$ is uniformly bounded below.
Thus $u=\lim_j u_j \in \PSH(X,\omega)$ with  $u\geq (1-b^{-1})\psi^- -b^{-1}C(b)$. 
 Since $(1-b^{-1})$ is arbitrarily small we infer, 
 using the main result of \cite{WN19},
 that $\int_X (\omega+dd^c u)^n=\int_X \omega^n$, hence $u\in \mathcal{E}(X,\omega)$. 
 Uniqueness follows from the comparison principle.
\end{proof}

\subsubsection{Uniform estimates}

Assume $\varphi_0 \in \PSH(X,\omega)$ is locally bounded away from some point $a \in X$,
 and  
 $$
 \varphi_0= \log (|z_1|^{2\alpha_1}+\cdots+|z_n|^{2\alpha_n})+r_0,
 $$ 
 in a local chart where $a=0$,  $\alpha_j >0$, and $r_0$ is smooth.
 It is classical that $c(\varphi_0)=\sum_j \alpha_j^{-1}$.

\begin{prop} \label{pro:c0toric}
	The logarithmic singularity of $\varphi_t$ is provided by $(1-ct)\varphi_0$. Moreover
	$$
	(1-ct)\varphi_0+ct\rho+ n(t\log t-t) -C(t+1)\leq \varphi_t\leq (1-ct)\varphi_0+C(t+1),
	$$  
	where $C>0$, $c=c(\varphi_0)$, and $\rho$ is the unique solution of \eqref{eq_crit}.
	\end{prop}	
	
	When $\alpha_1=\cdots=\alpha_n=1$ this result follows from Proposition \ref{pro:c0isolatedhomog},
	as one can check that in this case $c\rho \sim -2 \log (-\log |z|^2)$.

	\begin{proof}
We approximate $\varphi_0$ by $\varphi_{0,\varepsilon}= \log (e^{\varphi_0}+\varepsilon)+r_0$, $\varepsilon>0$. 
We let the reader check that 
		\begin{flalign*}
(\omega + dd^c \varphi_{0,\varepsilon})^n 
&\lesssim \left(\omega + \frac{dd^c e^{\varphi_0}}{e^{\varphi_0}+\varepsilon}\right)^n 
\sim  \sum_J \frac{|z_{J}|^{2(\alpha_J-1)}}{(e^{\varphi_0}+\varepsilon)^m}  dV_X.  
		\end{flalign*}
		Here, $J=(j_1,...,j_m)$ with $m\leq n$ and $|z_J|^{\alpha_J}= \prod_{j\in J} |z_j|^{\alpha_j}$.

We claim that 
		\[
		\frac{|z_1|^{2(\alpha_1-1)}\cdots |z_m|^{2(\alpha_m-1)}}{(e^{\varphi_0}+\varepsilon)^m}  \leq e^{-c\varphi_{0,\varepsilon}}\sim \frac{1}{(|z_1|^{2\alpha_1}+\cdots+|z_n|^{2\alpha_n}+\varepsilon)^c}.
		\]
		The latter is equivalent to
		\begin{equation} \label{eq:toric}
		\left(|z_1|^{2(\alpha_1-1)}\cdots |z_m|^{2(\alpha_m-1)}\right)^{\frac{1}{m-c}}  \leq  (|z_1|^{2\alpha_1}+\cdots+|z_n|^{2\alpha_n}+\varepsilon).
		\end{equation}
		%Fixing $J$, the inequality corresponding to $J$ is a consequence of 
		%\[
		%q_1x_1+...+q_mx_m \geq x_1^{q_1}\cdots x_m^{q_m}, \;  \text{where}\; 
Set $x_j:= |z_j|^{2\alpha_j}$, $q_j=(1-\alpha_j^{-1})/(m-c)$ and $c_J=\sum_{j\in J} \alpha_j^{-1}\leq c$.
The concavity of $\log$ yields,
since $p_j:=(1-\alpha_j^{-1})/(m-c_J)\leq q_j$ is such that $\sum_{j\in J} p_j=1$,
		\[
x_1^{q_1}\cdots x_m^{q_m} \leq 	x_1^{p_1}\cdots x_m^{p_m} \leq 	p_1x_1+...+p_mx_m \leq x_1+...+x_m,
		\]
		when $|z_j|$ is less than $1$, showing \eqref{eq:toric}.
		
 	It follows that $u_t=(1-ct) \varphi_{0,\varepsilon}+Ct$ is a supersolution of the smooth approximating
 	flow, since
 	$(\omega+dd^c u_t)^n \leq (\omega + dd^c \varphi_{0,\varepsilon})^n  \leq e^{\partial_t u_t} dV_X$,
 	for $C$ large enough.
		Letting $\varepsilon \rightarrow 0$, we conclude that $\varphi_t \leq (1-ct)\varphi_{0}+Ct +r_0 $.
		The upper bound follows since $r_0$ is bounded.

\smallskip

Consider now $\psi_t(x)=(1-ct)\varphi_0 + ct \rho +n(t\log t-t)-Ct+r_0$.
A direct computation shows that 
		\[
		(\omega+dd^c \psi_t)^n \geq c^n t^n (\omega+dd^c \rho)^n \geq e^{c\rho -c\varphi_0 + n \log t + n \log c}\omega_X^n\geq e^{\dot \psi_t} dV_X.
		\]
		Thus  $\psi_t$ is a subsolution of the flow with $\psi_0=\varphi_0$, 
		%hence $\psi_t\leq \varphi_t$.
		and the desired lower bound follows.
	\end{proof}

\subsection{Combination of singularities}

We consider here an initial potential $\varphi_0$ such that
\begin{itemize}
\item $\f_0=m \log |\sigma|^2+ \gamma u_0$ for some irreducible smooth divisor $D=(\sigma=0)$, some $m,\gamma>0$,
and a quasi-psh function $u_0$ which is smooth but at one point $a \in X$;
\item $u_0$ has an isolated {\it homogeneous logarithmic singularity}
at $a$: 
in a local chart when $a=0 \in \B^n \subset \C^n$,  
$u_0(z)= \log |z|^2+v_0(z)$, where $v_0$ is a smooth function.
\end{itemize}

When $a \notin (\sigma=0)$ the asymptotic behavior is a simple combination of the findings of \cite{DGL26}
together with those from Section \ref{sec:isolatedhomogeneous}, so we assume in what follows that
$a \in (\sigma=0)$.

\subsubsection{First order asymptotics}

\begin{lem}
One has  $\la(\f_0)=\max\left(m, \frac{\gamma+m}{n}\right)$.
\end{lem}

\begin{proof}
Let $\pi:Y \rightarrow X$ denote the blow up of $X$ at the point $a$. We obtain
$$
\int_X e^{-c \f_0} dV_X=\int_Y \frac{|\sigma_E|^{2(n-1)}}{|\sigma_E|^{2c(m+\gamma)} |\tilde{\sigma}|^{2cm}} dV_Y,
$$
where $(\sigma_E=0)$ denotes the exceptional divisor, $\tilde{D}=(\tilde{\sigma}=0)$ denotes the strict transform of $D$,
and $dV_Y$ is a smooth volume form on $Y$.
Since $E$ and $\tilde{D}$ have simple normal crossings, the conclusion follows from \eqref{int formula} since in this case $m_1=m, m_2=\gamma+m, b_1=0, b_2=n-1$.
\end{proof}

\begin{lem}\label{lem_combSing}
For all $(t,x)\in \left(0,\min \left(m, \frac{\gamma}{n-1}\right)\right)\times X$,
$$
\f_t \leq (m-t) \log |\sigma|^2+ (\gamma -(n-1)t) u_0+Ct,
$$
for some uniform constant $C>0$.
\end{lem}

\begin{proof}
Observe that $t<\min(m,\frac{\gamma}{n-1}) \Rightarrow t<\min (m,\frac{\gamma+m}{n})$.
 For any $\e>0$, we set
 $$
 \p_{t,\e}=(m-t) \log(|\sigma|^2+\e^2)+(\gamma-(n-1)t)  \log(|z|^2+\e^2)+v_0.
 $$
An elementary computation gives 
 $$
 e^{\partial_t  \p_{t,\e}} dV_X \leq \frac{e^C}{(|\sigma|^2+\e^2)(|z|^2+\e^2)^{n-1}} dV_X,
 $$
 for $C>0$ large enough.
  On the other hand 
  $$
  \omega+dd^c \p_{t,\e} \leq B \omega+ B\frac{i \partial \sigma \wedge \overline{ \partial \sigma}}{|\sigma|^2+\e^2} + \frac{B \omega}{|z|^2+\e^2}
  $$
Therefore, using that $i \partial \sigma \wedge \overline{ \partial \sigma} \wedge \omega^{n-1}\geq c\omega^n$ and $|\sigma|^2\leq |z|^2$, we obtain 
  $$
 (\omega+dd^c  \p_{t,\e})^n \leq \frac{B'}{(|\sigma|^2+\e^2)(|z|^2+\e^2)^{n-1}} dV_X,
  $$
  showing that $\p_{t,\e} +Ct$ is a smooth supersolution of the flow starting at  $\f_{0,\e}$.
  The   parabolic maximum principle ensures that
  $\f_{t,\e} \leq \p_{t,\e}$ and the conclusion follows by letting $\e \rightarrow 0$.
\end{proof}

\subsubsection{Uniform estimates}

\begin{prop} \label{pro:c0combine}
Set $$\psi_t:=  (m-t) \log |\sigma|^2+ (\gamma -(n-1)t) u_0  +2t (-\log (-\log|\sigma|^2)) + 2t (-\log(-u_0)).
$$ For all $(t,x) \in \left(0,\min(m, \frac{\gamma}{n-1}) \right)\times X$ we have
$$h(t) \leq \f_t - \psi_t \leq Ct,$$ 
where $C>0$ and $h$ is a log-Lipschitz function such that $h(0)=0$. 
Thus for all $x \in X$,
$$
\lambda(\f_t,x)=\max \left( \lambda(\f_0,x)-t, 0 \right).
$$
\end{prop}

\begin{proof}
We need to show that $\psi_t \pm C(t)$ provides sub/super solutions to the flow. To do this we only need to compute near $a$. 
Set $L=\log |z|^2$ and observe that
 $$
 \omega+dd^c \psi_t \leq c \left( \omega+dd^c L
+\frac{d\sigma \wedge d^c \sigma}{|\sigma|^2 (-\log |\sigma|^2)^2}
+\frac{dL \wedge d^c L}{(-L)^2} \right).
 $$
 Using that $(dd^c L)^n=0$ in $X \setminus D$ and that $d\sigma \wedge d^c \sigma$ and $dL \wedge d^c L$
 have rank $1$, we observe that, 
 $(dd^c L)^{n-2} \wedge \frac{dL \wedge d^c L}{(-L)^2}  \wedge \frac{d\sigma \wedge d^c \sigma}{|\sigma|^2 (-\log |\sigma|^2)^2}$
 is the dominant term in $(\omega+dd^c \psi_t)^n$, hence
 $$
 (\omega+dd^c \psi_t)^n \leq  \frac{c'}{|\sigma|^2(-\log|\sigma|^2)^2 |z|^{2(n-1)} (-\log|z|^2)^2} dV_X
 \leq e^{\partial_t(\psi_t+Ct)} dV_X.
 $$
The above computation shows that $\psi_t+Ct$ is a supersolution of the flow. 
Lemma \ref{lem_combSing} ensures
 $$
 \psi_t\geq \varphi_t -2t\log (-\log|\sigma|^2) -2t\log(-u_0)
 $$ 
 with $-2\log (-\log|\sigma|^2) $ and $-2\log(-u_0)$ belonging to the class $\mathcal{E}(X, \omega)$.
 Since $\psi_t$ is smooth outside $D$, it follows therefore from Theorem \ref{thm:pcp}
 that $\f_t \leq \p_t+Ct$.

 \smallskip
 
We now prove the lower bound. A similar computation yields
 $$
 \omega+dd^c \psi_t \geq (\gamma-(n-1)t) dd^c L + 2t \frac{d\sigma \wedge d^c \sigma}{|\sigma|^2 (-\log |\sigma|^2)^2}
+2t\frac{dL \wedge d^c L}{(-L)^2}
 $$
 hence
 \begin{eqnarray*}
  (\omega+dd^c \psi_t)^n &\geq & 4n(n-1)t^2 (\gamma-(n-1)t)^{n-2} (dd^c L)^{n-2} \wedge \frac{d\sigma \wedge d^c \sigma}{|\sigma|^2 (-\log |\sigma|^2)^2}
\wedge \frac{dL \wedge d^c L}{(-L)^2} \\
&= & c_n t^2 (\gamma-(n-1)t)^{n-2} \frac{\omega^n}{|\sigma|^2(-\log|\sigma|^2)^2 |z|^{2(n-1)} (-\log|z|^2)^2} \\
& \geq & t^2 (\gamma-(n-1)t)^{n-2} e^{\partial_t(\psi_t-Ct)} dV_X \\
&=& e^{\partial_t(\psi_t+h(t))} \omega^n,
 \end{eqnarray*}
 where $h$ is the function such that $h(0)=0$
 and $e^{h'(t)}=t^2 (\gamma-(n-1)t)^{n-2}e^{-C}$.
\end{proof}

\subsubsection{${\mathcal C}^2$-estimates} \label{sec:combineC2}

One can then argue as in the proof of Theorem \ref{thm:c2hyp} and 
show that $\omega_t$ is quasi-isometric to the metric 
$\omega_D+\beta_a$, where $\omega_D$ denotes the Poincaré metric in $X \setminus D$,
while $\beta_a$ is the model metric from Section \ref{sec:isolatedhomogeneous}.
We leave the details to the interested reader.

\end{document}